\documentclass[11pt,twoside]{amsart}

\usepackage{latexsym}
\usepackage{amsmath}
\usepackage{amsthm}
\usepackage{amssymb}
\usepackage{vmargin}
\usepackage{stmaryrd}
\usepackage{euscript}
\usepackage{mathrsfs}
\usepackage{amscd}
\usepackage[all]{xy}
\usepackage{xr}

\newtheorem{theorem}{Theorem}[section]
\newtheorem{proposition}[theorem]{Proposition}
\newtheorem{corollary}[theorem]{Corollary}

\newtheorem{lemma}[theorem]{Lemma}
\newtheorem*{theorem*}{Theorem}
\newtheorem*{proposition*}{Proposition}
\newtheorem*{corollary*}{Corollary}
\newtheorem*{lemma*}{Lemma}
\newtheorem*{conjecture*}{Conjecture}

\theoremstyle{definition}
\newtheorem{definition}[theorem]{Definition}

\newtheorem*{definition*}{Definition}

\theoremstyle{remark}
\newtheorem{example}[theorem]{Example}

\newtheorem{remark}[theorem]{Remark}
\newtheorem{remarks}[theorem]{Remarks}

\newtheorem*{example*}{Example}
\newtheorem*{examples*}{Examples}
\newtheorem*{remark*}{Remark}
\newtheorem*{remarks*}{Remarks}
\newtheorem*{exercise*}{Exercise}
\newtheorem*{property*}{Property}
\newtheorem*{properties*}{Properties}

\newcommand\la{\leftarrow}

\newcommand\id{\mathrm{id}}

\newcommand\ten{\otimes}

\newcommand\CC{\mathrm{C}}
\newcommand\DD{\mathrm{D}}

\renewcommand\H{\mathrm{H}}

\newcommand\HH{\mathrm{HH}}
\newcommand\HC{\mathrm{HC}}
\newcommand\HP{\mathrm{HP}}
\newcommand\HN{\mathrm{HN}}

\newcommand\Z{\mathbb{Z}}
\newcommand\Q{\mathbb{Q}}

\newcommand\Cx{\mathbb{C}}

\newcommand\bH{\mathbb{H}}

\newcommand\bL{\mathbb{L}}

\newcommand\C{\mathcal{C}}

\newcommand\cD{\mathcal{D}}

\newcommand\cH{\mathcal{H}}

\newcommand\cL{\mathcal{L}}

\newcommand\cN{\mathcal{N}}

\newcommand\cO{\mathcal{O}}
\newcommand\cP{\mathcal{P}}

\newcommand\cW{\mathcal{W}}

\renewcommand\O{\mathscr{O}}

\newcommand\sA{\mathscr{A}}

\newcommand\sF{\mathscr{F}}

\newcommand\sI{\mathscr{I}}

\newcommand\sN{\mathscr{N}}
\newcommand\sO{\mathscr{O}}

\newcommand\fM{\mathfrak{M}}
\newcommand\fN{\mathfrak{N}}

\renewcommand\L{\Lambda}

\newcommand\CAlg{\mathrm{CAlg}}

\newcommand\Hom{\mathrm{Hom}}
\newcommand\Map{\mathrm{Map}}

\newcommand\HHom{\underline{\mathrm{Hom}}}

\newcommand\Ext{\mathrm{Ext}}
\newcommand\EExt{\mathbb{E}\mathrm{xt}}

\newcommand\cone{\mathrm{cone}}
\newcommand\cocone{\mathrm{cocone}}

\newcommand\per{\mathrm{per}}

\newcommand\im{\mathrm{Im\,}}

\newcommand\ch{\mathrm{ch}}
\newcommand\Td{\mathrm{Td}}

\newcommand{\AH}{\mathrm{AH}}

\newcommand\Spec{\mathrm{Spec}\,}

\newcommand\Set{\mathrm{Set}}

\newcommand\LLim{\varinjlim}

\newcommand\into{\hookrightarrow}
\newcommand\onto{\twoheadrightarrow}

\newcommand\xra{\xrightarrow}
\newcommand\xla{\xleftarrow}

\newcommand\pr{\mathrm{pr}}

\newcommand\alg{\mathrm{alg}}

\newcommand\bt{\bullet}
\newcommand\by{\times}

\newcommand\Perf{\mathrm{Perf}}

\newcommand\Symm{\mathrm{Symm}}

\newcommand\GL{\mathrm{GL}}

\newcommand\et{\acute{\mathrm{e}}\mathrm{t}}

\newcommand\an{\mathrm{an}}

\newcommand\Tot{\mathrm{Tot}\,}

\newcommand\tr{\mathrm{tr}}
\newcommand\ev{\mathrm{ev}}

\renewcommand\alg{\mathrm{alg}}
\newcommand\red{\mathrm{red}}

\newcommand\dR{\mathrm{dR}}
\newcommand\DR{\mathrm{DR}}

\newcommand\op{\mathrm{opp}}

\newcommand\co{\colon\thinspace}

\newcommand\oR{\mathbf{R}}

\newcommand\oL{\mathbf{L}}

\newcommand\uleft\underleftarrow
\newcommand\uline\underline
\newcommand\uright\underrightarrow

\begin{document}
\title{Semiregularity as a consequence of Goodwillie's theorem 
}

\author{J.P.Pridham}
\thanks{This work was supported by  the Engineering and Physical Sciences Research Council [grant number   EP/I004130/1].}

\begin{abstract}
  We realise Buchweitz and Flenner's semiregularity map (and hence \emph{a fortiori} Bloch's semiregularity map) for a smooth variety $X$ as the tangent of a generalised Abel--Jacobi map on the derived moduli stack of perfect complexes on $X$. The target of this map is an analogue of Deligne cohomology defined in terms of cyclic homology, and Goodwillie's theorem on nilpotent ideals ensures that it has the desired tangent space  (a truncated de Rham complex).

Immediate consequences are the semiregularity conjectures: that the semiregularity maps annihilate all obstructions, and that if $X$ is deformed, semiregularity measures the failure of the Chern character to remain a Hodge class. This gives rise to reduced obstruction theories of the type featuring in the study of reduced Gromov--Witten and Pandharipande--Thomas invariants. We also give generalisations allowing  $X$ to be singular, and even  a derived stack. 
\end{abstract}

\maketitle

\section*{Introduction}

In \cite{blochSemiregularity}, Bloch defined a semiregularity map
\[
 \tau \co \H^1(Z, \sN_{Z/X}) \to \H^{p+1}(X, \Omega^{p-1}_X)
\]
for every local complete intersection $Z$ of codimension $p$ in a smooth proper complex variety $X$, and  showed that curvilinear obstructions lie in the kernel of $\tau$. He also showed that if $X$ is deformed, then $\tau$ measures the curvilinear obstruction to $[Z]$ remaining a Hodge class, 
 and conjectured that these statements should hold for all obstructions, not just curvilinear ones.

In \cite{BuchweitzFlenner}, Buchweitz and Flenner extended $\tau$ to give maps
\begin{align*}
 \sigma_q \co \Ext^2_{\sO_X}(\sF,\sF) &\to \H^{q+2}(X, \Omega^q_X)
 \end{align*}
for any perfect complex $\sF$ on $X$, given (up to scalar) by composing with the $q$th power of the Atiyah class then taking the trace.
They then showed that curvilinear obstructions to deforming $\sF$ lie in the kernel of $\sigma_q$. If $X$ is  allowed to deform, they  showed the same holds provided $\ch_p(\sF)$ deforms as  a Hodge class, with  consequences for the variational Hodge conjecture \cite[\S 5]{BuchweitzFlenner}.  

Obstruction spaces feature in the  construction of  virtual fundamental classes in enumerative geometry, used to construct Gromov--Witten and similar invariants. Often the natural obstruction spaces are too large, killing the na\"ively defined invariants, but the semiregularity conjectures yield smaller reduced obstruction theories $\ker(\tau)$ and $\ker(\sigma_q)$ in many cases, paving the way for non-trivial reduced invariants to be defined. 

Buchweitz and Flenner \cite[\S 1, p.~138]{BuchweitzFlenner} also conjectured that the semiregularity map should arise as a morphism of obstruction spaces associated  to a
morphism of deformation theories, then speculated that this morphism would most likely take the form of a 
generalised Abel--Jacobi map from the deformation groupoid to an intermediate Jacobian or to Deligne cohomology. 
The underlying idea is that for a deformation $\tilde{\sF}$ of $\sF$, we must have $\ch(\sF)=\ch(\tilde{\sF})$ because the cohomological Chern character takes rational values. The homotopy between cycles representing $\ch_{q+1}(\sF)$ and $\ch_{q+1}(\tilde{\sF})$  should then be given by  $\sigma_q(\tilde{\sF})$. 

When seeking functorial obstruction theories, one is naturally drawn to derived deformation theory, which generates obstruction spaces as higher tangent spaces (e.g.\ see Lemma \ref{obs}), and guarantees functoriality of obstruction maps. In \cite{manlie,IaconoManettiSemireg},  Manetti and Iacono  constructed explicit infinitesimal derived Abel--Jacobi maps by $L_{\infty}$ methods, proving Bloch's first semiregularity conjecture in the case where $Z$ is smooth, and then for complete intersections of hypersurfaces.\footnote{Nearly a decade after this was first written, Bandiera, Lepri and Manetti \cite{BandieraLepriManettiSemireg} used Chern--Simons classes to recover the main results of this paper  in the absolute case $R=\Cx$ for smooth proper varieties $X$, including the first semiregularity conjecture.} 
Like \cite{BuchweitzFlenner},  \cite{IaconoManettiSemireg} identified $\bH^{2p}(X, \Omega_X^{< p})$ as a more natural target for the semiregularity map than  $\H^{p+1}(X, \Omega^{p-1}_X)$. Other 
work such as \cite{STV, KoolThomas1}   focuses on the case $p=1$, where this discrepancy does not arise.\footnote{In that special case, the derived Picard stack already provides a suitable target (since the trace $\sigma_0$ is an isomorphism for line bundles), with the Abel--Jacobi map simply corresponding to the derived determinant. 
}

In this paper, we construct a morphism of derived deformation theories of the form envisaged in \cite{BuchweitzFlenner}, but to a slightly different target. This leads to the following theorem
, which proves and generalises the  conjectures of \cite{BuchweitzFlenner} and hence \cite{blochSemiregularity}, showing that the semiregularity map measures the failure of the unique horizontal lift of the Chern character to remain in $F^p$: 
\begin{theorem*}  
 Take a local Artinian $\Cx$-algebra $A$, a smooth  morphism $X \to \Spec A$ of Artin stacks and square-zero ideal $I \subset A$ with quotient $B= A/I$. Then for  any perfect complex $\sF$ over  $X':=X\ten_AB$, with obstruction $o(\sF) \in \EExt^{2}_{\sO_{X'}}(\sF,\sF\ten_B I)$  to deforming $\sF$ to a complex of $\sO_X$-modules, the 
 image of the 
 Chern character $\ch_p(\sF)$ under the map
\[
 \H^{2p}(X'(\Cx)_{\an},\Q) \cong \H^{2p}(X(\Cx)_{\an},\Q) \to \H^{2p}(X(\Cx)_{\an},A) \cong \H^{2p}(X,\Omega^{\bt}_{X/A}) 
 \]
lies in $F^p\H^{2p}(X,\Omega^{\bt}_{X/A})$
 if and only if $o(\sF)$ maps to zero under the composite map
 \[
  \EExt^{2}_{\sO_{X'}}(\sF,\sF\ten_B I) \xra{\sigma_{p-1} } \H^{p+1}(X,I\Omega^{p-1}_{X/A}) \to \bH^{2p}(X,\Omega^{<p}_{X/A}).
 \]
\end{theorem*}
In fact, we establish 
a more  general statement, Corollary \ref{horizobscor2}, using derived differential forms to remove the smoothness hypothesis, and allowing derived objects. Note that if the family $X$ is constant over $A$ (i.e.\ $X\cong X_0 \by \Spec A$)
then the condition $\ch_p(\sF) \in F^p$ is automatically satisfied and the obstruction $o(\sF)$ always maps to zero (Remark \ref{splitrmk}).

The theorem produces reduced obstruction theories for the stable pairs and stable curves featuring in  the study of Pandharipande--Thomas and Gromov--Witten  invariants (Remarks \ref{horproperrk}  and Remark \ref{GWrk}). The latter follows 
by considering perfect complexes of the form $\oR f_*\O_Z$ to produce reduced obstruction theories  for proper  morphisms $f\co Z\to X$. As a special case,  this establishes Bloch's   semiregularity conjectures in the generality envisaged (Remark \ref{Blochrk}). Since we can apply the theorem to gerbes, it also leads to reduced obstructions for $\mu$-twisted sheaves (Remark \ref{mutwistrmk}).

The previous theorem is a consequence of the following more general result, which only involves cohomology groups of algebraic, not analytic or topological, origin (cf. Corollary \ref{gensemiregcor} for the derived generalisation not requiring smoothness):  
\begin{theorem*} 
Take a smooth morphism $f \co X \to S$ of Artin stacks 
 over $\Q$, with  a  closed immersion $S' \into S$ defined by a nilpotent ideal $\sI$. 
Then for $X':=X\by_{S}S'$, the  Chern character refines to give maps $\Xi_p$ from  $K_0(X')$ to the vector spaces 
\begin{align*}
&\bH^{2p}(X, \Omega^{\bt}_{X/S}\by_{\Omega^{\bt}_{X'/S'}}F^p\Omega^{\bt}_{X'/S'}) \\
 &= \bH^{2p}(X, \sI\sO_{X} \xra{d} \sI\Omega^1_{X/S} \xra{d} \ldots \xra{d} \sI\Omega^{p-1}_{X/S} \xra{d} \Omega^p_{X/S}     \xra{d} \Omega^{p+1}_{X/S}  \xra{d}\ldots ).
\end{align*}

If $\sI^2=0$, then for  any perfect complex $\sF$ over  $X'$  the 
obstruction to lifting $\Xi_p(\sF)$ to $\bH^{2p}(X, F^p\Omega^{\bt}_{X/S})$ 
 is given by applying the composite  map 
 \[
  \EExt^{2}_{\sO_{X'}}(\sF,\sF\ten_{\sO_{S'}} \sI) \xra{\sigma_{p-1} } \bH^{p+1}(X,\sI\Omega^{p-1}_{X/S}) \to \bH^{2p}(X,\sI\Omega^{<p}_{X/S}) 
 \]
 to  the obstruction $o(\sF)$  to deforming $\sF$ to a complex of $\sO_X$-modules. 
\end{theorem*}

The crucial  observation enabling our construction $\Xi$  is that if we modify  Deligne cohomology slightly, replacing 
rational Betti cohomology with any other   cohomology theory which is formally \'etale (i.e.\ invariant under nilpotent thickenings), 
then  the obstruction spaces  are unchanged. The theory we work with is  Hartshorne's algebraic de Rham cohomology $\DR^{\alg}$ in the guise of derived de Rham cohomology, which is formally \'etale by Goodwillie's theorem on nilpotent ideals \cite
{goodwillieHCderivations}. Our map $\Xi$ is then induced from the Goodwillie--Jones Chern character $\ch^-$.

Given a smooth morphism $X \to \Spec R$ over $\Cx$,  we  set $J^p(X/R,\Cx)[2p]$  to be the  cocone (i.e.\ shifted cone or homotopy fibre) of 
\[
 \DR^{\alg}(X/\Cx) \to \oR \Gamma(X,  \Omega^{< p}_{X/R}),   
\]
noting that in general $X$ will not be smooth over $\Cx$. This definition also adapts in the obvious way to any  base $\Q$-algebra $k$ in place of $\Cx$, and admits  further generalisations to derived stacks with no smoothness hypothesis  (Definition \ref{JXpglobaldef}). 

To establish existence of the Abel--Jacobi map and functoriality, we reformulate   in terms of cyclic homology. Derived de Rham cohomology is isomorphic  to periodic cyclic homology $\HP$, giving
\[
\prod_p J^p(X/R,k) \simeq \HP^k(X)\by^h_{\HP^R(X)}\HN^R(X),
\]
for negative cyclic homology $\HN$.
The generalised Abel--Jacobi maps 
\[
\Xi_p \co K(X) \to J^p(X/R,k)
\]
 are then induced from the Goodwillie--Jones Chern character
$
 \ch^-\co K(X) \to \HN^k(X). 
$
 
Now, 
$\Xi_p$ restricts from $K$-theory to a map on the $\infty$-groupoid $\Perf(X)$ of perfect complexes on $X$. As we change the base, setting $X_A:= X\ten_RA$ and $J_X^p(A,k):= J^p(X_A /A,k)$, this  gives us  morphisms
\[
\Xi_p \co  \Perf( X_A) \to J^p_X(A,k),
\]
functorial in 
simplicial $R$-algebras $A$.

Goodwillie's Theorem on nilpotent ideals implies that $\HP^k(X_A)$ is formally \'etale as a functor in $A$, so $J_X^p(-,k)$ 
has  
the same derived tangent space as  cyclic homology with a degree shift.  
On derived tangent spaces,  $\Xi_p$ thus  induces  maps
\[
 \xi_p\co \EExt^r_{\sO_{X_A}}(\sF,\sF\ten_AM) \to \bH^{2p-2+r}(X, (\O_X \to \ldots \to \Omega^{p-1}_X)\ten_R M),
\]
for $A$-modules $M$;  in Proposition \ref{AJtgtprop}, we show that $\xi_p$ is just  the $(p-1)$th component of the  Lefschetz map $\cL$ of \cite{BresslerNestTsygan}, and hence (Remark \ref{cfBF}) equivalent to the semiregularity map $\sigma_{p-1}$ of \cite{BuchweitzFlenner}.

The \'etale 
hypersheaves $\Perf_X$ and $J_X^p(-,k)$ satisfy homotopy-homogeneity\footnote{Later terms for this concept are  infinitesimal cohesiveness on one factor and (on Artinian input)  being a formal moduli problem, although the property neither associates genuine moduli problems to such functors nor is automatically satisfied by such.
},
a left-exactness property analogous to Schlessinger's conditions, which in particular gives a functorial identification of higher tangent spaces with obstruction spaces. Since $J_X^p(-,k)$ has the same obstruction space as Deligne cohomology, 
the map $\Xi_p$ thus fully realises the hope expressed in \cite[\S 1]{BuchweitzFlenner}.

Given a square-zero extension $e\co A \to B$ of simplicial algebras with kernel $I$, and a perfect complex $\sF$ on $X_B$, the obstruction $o_e(\sF)$ to lifting $\sF$ to $X_A$ lies in $\Ext^2_{\sO_{X_B}}(\sF, \sF\ten_B I)$. 
Derived functoriality and homotopy-homogeneity then ensure (Corollary \ref{AJtgtcor2}) that  the obstruction to lifting 
$\Xi_p(\sF)$ from $\H_0J_X^p(B,k)$ to $\H_0J_X^p(A,k)$ is just 
\[
 \xi_{p}(o_e(\sF)) \in \bH^{2p}(X, (\sO_{X} \to \Omega_{X/R}^1\to \ldots \to \Omega^{p-1}_{X/R})\ten_RI),
\]
 leading to the theorems above (Corollaries \ref{horizobscor2} and \ref{gensemiregcor}), corresponding to the choices $k=\Cx$ and $k=R=A$, respectively. 
Our formulation in terms of cyclic homology and the Lefschetz map $\cL$ also extends these results to certain non-commutative spaces (Corollary \ref{orlovcor}), where one effectively has to consider all values of $p$ simultaneously.

%
\medskip
%

I am indebted to  Timo Sch\"urg for
bringing \cite{BuchweitzFlenner} to my attention. I would also like to thank Barbara Fantechi, Richard Thomas and Daniel Huybrechts for helpful comments, and the anonymous referee for diligently suggesting improvements.

\tableofcontents

\subsection*{Notation and conventions}


The Dold--Kan correspondence gives an equivalence of categories between
simplicial abelian groups and non-negatively graded chain complexes, with homotopy groups corresponding to homology groups, and we will pass between these categories without further comment. 

Given a chain complex $V$, we will write $V[n]$ for the chain complex  given by $V[n]_i= V_{n+i}$; beware that this is effectively opposite to the standard convention for cochain complexes. 
We also write $\tau_{\ge 0}V$ for the good truncation \cite[Truncations 1.2.7]{W} of $V$ in non-negative chain degrees, which we can then regard as a simplicial abelian group by the Dold--Kan correspondence above.

For a morphism $f\co V \to W$ of chain complexes, $\cocone(f)$ will denote the shifted cone, a model for the homotopy fibre of $f$, which fits into an exact triangle
\[
 \cocone(f) \to V \to W \to \cocone(f)[-1].
\]

\begin{definition}\label{HCdef}
Given a commutative ring $A$ and a flat $A$-algebra $E$, write $\HC^A(E)$ (resp. $\HN^A(E)$, resp. $\HP^A(E)$, resp. $\HH^A(E)$) for the chain complex associated to cyclic (resp. negative cyclic, resp. periodic cyclic, resp. Hochschild) homology of $E$ over $A$, as in \cite[\S 9.6]{W}. 

Given a simplicial commutative ring $A$ and a  simplicial $A$-algebra $E$ with each $E_n$ flat over $A_n$,
together with  a homology theory $\cH$ as in the previous paragraph,  define the complex $\cH^A(E)$  by first forming the simplicial chain complex given by $\cH^{A_n}(E_n)$ in level $n$, then taking the product total complex.  
\end{definition}

\begin{remark}
When working with cyclic homology, it is usual to fix a base ring and to omit it from the notation. Since varying the base  will be  crucial to our constructions, we have introduced the superscript $A$ above. Also beware that the cohomology theories $\HN$ and $\HP$ are frequently denoted by $\HC^-$ and $\HC^{\mathrm{per}}$ in the literature, and that our complexes are related to cyclic homology groups by
\[
 \HC^A_i(E,M):= \H_i\HC^A(E,M)
\]
etc.  In the notation of \cite[Ch. 9]{W}, the complexes $\HH, \HC, \HN, \HP$ are denoted by $\CC\CC_*^h, \Tot \CC\CC_{**},  \Tot^{\Pi} \CC\CC_{**}^N, \Tot^{\Pi} \CC\CC_{**}^P$.

 When $A$ is a discrete ring, note that the complexes above are those studied in \cite{goodwillieChern}.    
\end{remark}

\begin{definition}
 As in \cite[\S 9.8.2]{W}, when the $A$-algebra $E$ is commutative each of the homology theories $\cH$ above admits a Hodge decomposition, which we denote by
\[
 \cH^A(E)=\prod_{p \in \Z} \cH^A(E)^{(p)}.
\]
\end{definition}
Note that in the case of $\HC$, we have $\HC^A(E)^{(p)}=0$ for $p<0$, with degree bounds on the other terms (specifically  $\HC^A(E)^{(p)}$ concentrated in chain degrees $\ge p$) 
making the infinite product a direct sum in that case only.


For $E$ commutative, recall from \cite[\S 9.6.1]{W} that there are exact triangles (the  $\mathsf{SBI}$ sequences) 
\begin{align*}
 \HN^A(E)^{(p)} \xra{\mathsf{I}} \HP^A(E)^{(p)} \xra{\mathsf{S}} \HC^A(E)^{(p-1)}[-2] \xra{\mathsf{B}}  \HN^A(E)^{(p)}[-1]\phantom{,}\\
\HH^A(E)^{(p)} \xra{\mathsf{I}} \HC^A(E)^{(p)} \xra{\mathsf{S}} \HC^A(E)^{(p-1)}[-2] \xra{\mathsf{B}}  \HH^A(E)^{(p)}[-1],
\end{align*}
compatible with the projection map $\pi_{\HH} \co \HN^A(E)^{(p)}\to \HH^A(E)^{(p)}$ and $\mathsf{S}\co \HP^A(E)^{(p)} \to \HC^A(E)^{(p)}$. For $E$ non-commutative, these sequences still exist once we drop the superscripts $(p)$, there being no Hodge decomposition in this case. See Proposition \ref{cfDRHC} below for the relation of these sequences with the Hodge filtration on derived de Rham cohomology.

\begin{definition}
Given a simplicial ring $E$, we follow \cite{waldhausen} in writing $K(E)$ for the simplicial set constituting the $K$-theory space of $E$ (the $0$th part of the $K$-theory spectrum). This is an infinite loop space with $\pi_iK(E)= K_i(E)$. 
\end{definition}

\section{The Abel--Jacobi map for rings}\label{AJrings}

Fix 
a simplicial commutative $\Q$-algebra $R$ and a simplicial associative  $R$-algebra $O(X)$,  which need not be commutative. 
Assume that 
each $O(X)_n$ is flat as an $R_n$-module.
 We will write $F(X):= F(O(X))$ when  $F$ is a functor such as $K, \HP,\HC,\HN,\HH$.

 Write $s\CAlg_R$ for the category of simplicial commutative $R$-algebras. We use its model structure induced from the Kan--Quillen model structure on simplicial sets. We call a functor from $s\CAlg_R$ to a model category homotopy-preserving if it preserves weak equivalences. 
 

\subsection{The Abel--Jacobi map}

\begin{definition}\label{JXdef}
Define  chain complexes 
\begin{align*}
 J(X/R):= \cocone( \HP^{\Q}(X) \xra{\mathsf{S}}  \HC^R(X)[-2]),\\
 J^p(X/R):= \cocone( \HP^{\Q}(X)^{(p)} \xra{\mathsf{S}}  \HC^R(X)^{(p-1)}[-2]),
 \end{align*}
the latter only being defined when $O(X)$ is commutative, with $J(X/R)=\prod_{p \in \Z} J^p(X/R) $. 

These are functors on the category of arrows $(R \to O(X))$ in $s\CAlg_{\Q}$, so  
we may then define  functors $J_X$ and $J_X^p$ from $s\CAlg_R$ to chain complexes 
by  setting $J_X(A):= J(X_A/A)$ and $J_X^p(A):= J^p(X_A/A)$, 
where $O(X_A):=O(X)\ten_RA$. 
\end{definition}
%
Note that the equivalence $\cocone(\HP^R(X) \xra{\mathsf{S}}  \HC^R(X)[-2])\simeq \HN^R(X)$ gives a  homotopy fibre product characterisation 
\[
J(X/R) \simeq \HP^{\Q}(X)\by^h_{\HP^R(X)}\HN^R(X);
\]
the motivation for this construction is that $\HN^R(X)$ behaves like the Hodge filtration over $R$, while $\HP^{\Q}(X)$ behaves in some respects like  Betti cohomology.

The Goodwillie--Jones Chern character
\[
 \ch^-\co K(X) \to \tau_{\ge 0}\HN^{\Q}(X) 
\]
of \cite[Theorem II.3.1]{goodwillieChern} (there denoted $\alpha$) or \cite[\S 5]{HoodJones}, or rather its promotion to a natural  $\infty$-transformation  as in \cite[\S III.2.3]{goodwillieChern}, 
then combines with the  
 map 
\[
\HN^{\Q}(X)  \to  \HP^{\Q}(X)\by^h_{\HP^R(X)}\HN^R(X),
\]
natural in $X$ and $R$,
to give a natural map
\[
 \Xi\co  K(X )\to  \tau_{\ge 0}J(X/R)
\]
in the $\infty$-category of simplicial sets,\footnote{Note that  we are here following our stated convention of reinterpreting chain complexes  as simplicial abelian groups. The natural $\infty$-transformation is given by a  zigzag alternating natural transformations with natural weak equivalences, which is thus a natural transformation of the associated simplicial functor on hammock localisations.} 
which we call the (generalised) Abel--Jacobi map.

\begin{definition}\label{perfdef} 
 Given a simplicial ring $S$, define $\Perf(S)$ to be the simplicial set  given by first forming the core (i.e.\ the subcategory of quasi-isomorphisms)    of the simplicial category $\cP\!\textit{erf}(S)$ of perfect $S$-modules in complexes, then taking the nerve; see for instance \cite[Definitions 2.8 and 2.29]{dmsch} or \cite[\S 1.3.7]{hag2}. This becomes a simplicial semiring with addition given by block sum and multiplication by tensor product. 
\end{definition}

\begin{definition}\label{AJdef}
By \cite[Theorem 2.3.2]{waldhausen} and \cite[Theorem 1.9.8]{ThomasonTrobaugh},  there is a natural map $\Perf(X) \to K(X)$. 
Composing this with the Abel--Jacobi map above  gives us a map 
\begin{align*}
 \Xi \co \Perf(X) &\to \tau_{\ge 0}J(X/R). 
 \end{align*}
 \end{definition}

\subsection{Homogeneity and obstructions}\label{hgsobsn}

Say that a map $A \to B$ in $s\CAlg_R$ is a nilpotent extension if it is levelwise surjective, with the kernel $I$  satisfying $I^n=0$ for some $n$, where both the kernel and its powers are defined levelwise. In other words, the maps $A_i \to B_i$ are all nilpotent surjections, with a common bound on the index of nilpotency.

\begin{definition}\label{hgsdef}
 We say that a homotopy-preserving functor  $F$ from  $s\CAlg_R$ to a model category $\C$ is homotopy-homogeneous if      for $A \to B$ a nilpotent extension in $s\CAlg_R$ and $C \to B$ any morphism,
the map
\[
 F(A\by_BC) \to F(A)\by^h_{F(B)}F(C)       
\]
(to the homotopy fibre product) is a weak equivalence. 
When $\C$ is the category of chain complexes, this is equivalent to saying that we have an exact triangle 
\[
 F(B)[1] \to  F(A\by_BC) \to F(A) \oplus F(C) \to F(B).      
\]
\end{definition}

\begin{definition}
 Define the simplicial set-valued functor $\Perf_X$ on $s\CAlg_R$ by $\Perf_X(A) := \Perf(O(X)\ten_R A)$.
\end{definition}

\begin{definition}\label{Tdef} 
 Given a homotopy-homogeneous functor $F$ from $s\CAlg_R$ to simplicial sets (with the Kan--Quillen model structure),
 an element $x \in F(A)_0$ and an $A$-module $M$ in simplicial abelian groups, define the tangent space
\[
T_x(F,M)        
\]
to be the homotopy fibre of $F(A\oplus M) \to F(A)$ over $x$, for the trivial square-zero extension $A \oplus M \to A$.     

Note that homotopy-homogeneity of $F$ ensures that $T_x(F,M)$ has an infinite loop space structure; see for instance  \cite[Lemma \ref{drep-adf}]{drep}. 
We thus 
define tangent cohomology to be the abelian groups $\DD^{n-i}_x(F,M):= \pi_iT_x(F,M[-n])$, which are well-defined by that lemma.
\end{definition}

\begin{definition}
 Define a square-zero extension $I \to A \to B$  in $s\CAlg_R $ to be a levelwise surjection $A \to B$ in $s\CAlg_R $ with kernel $I$ being a square-zero simplicial ideal.
\end{definition}

\begin{lemma}\label{obs}
Take a homotopy-preserving and homotopy-homogeneous simplicial  set-valued functor $F$ on $s\CAlg_R$ and    
  a square-zero extension $I \to A \xra{e} B $ in   $ s\CAlg_R $. Then there is a naturally  associated section $o_e \co FB \to F(B\oplus I[-1])$ in the $\infty$-category of simplicial sets, 
  such that for
 any $x \in (FB)_0$, the homotopy fibre $(FA)_x$ of $F(e)$ over $x$ is naturally homotopic to the space $\{o_e(x)\}\by^h_{T_x(F,I[-1])}\{0\}$ of paths from $0$ to $o_e(x)$.

In particular, we have a functorial obstruction
\[
 o_e(x) \in \DD^1_x(F,I),
\]
which is zero if and only if $[x]$ lies in the image of 
\[
 e_*\co\pi_0(FA)\to\pi_0(FB).
\]
\end{lemma}
\begin{proof}
 This is contained in \cite[Lemma \ref{drep-obs}]{drep} and its proof, in which the obstruction maps given here are just one term in a long exact sequence of homotopy groups. 
 
 Explicitly, the \v Cech nerve of $A$ over $B$ gives a bisimplicial $R$-algebra  by 
 sending $n$ to the $(n+1)$-fold fibre product of $A$ over $B$, 
 and taking the  diagonal gives us a simplicial $R$-algebra $\tilde{B}$. The natural map $\tilde{B} \to B$ is a square-zero extension with acyclic kernel $J$, where $J_n \cong (I_n)^{n+1}$; acyclicity follows because $J$ is the diagonal of a levelwise acyclic bisimplicial module. 
 The diagonal embedding gives a natural map $I \into J$, with the square-zero property ensuring that $I \subset \tilde{B}$ is a simplicial ideal.
We then have  an isomorphism 
 $\tilde{B}/I \cong B \oplus (J/I)$, and since $J$ is acyclic the  $B$-module $J/I$ is  naturally weakly equivalent to $I[-1]$. The obstruction map $o_e$ then comes from the zigzag of simplicial sets
 \[
  F(B) \xla{\sim} F(\tilde{B}) \to F(B \oplus (J/I)) \xra{\sim} F(B\oplus I[-1]).
 \]

Moreover, for the zero section $B \to B \oplus J/I$, we have $\tilde{B}\by_{B\oplus (J/I)}B \cong A$ and hence $F(A)\simeq F(B)\by^h_{o_e, F( B\oplus I[-1]), 0} F(B)$, giving  the required fibre sequence on taking homotopy fibres over $x \in F(B)$.
 \end{proof}

The following is well-known (see for instance  \cite[Theorem \ref{dmsch-representdmod}]{dmsch};  although stated for $O(X)$  commutative, the proof works verbatim in our generality):
\begin{lemma}
The functor $\Perf_X$ is homotopy-preserving and homotopy-homogeneous.
At $\sF \in \Perf_X(A)$,  the tangent space $T_{\sF}( \Perf_X,M)$ is $\tau_{\ge 0} (\oR \HHom_{O(X)\ten_RA}(\sF,\sF\ten_A M)[-1])$, so the tangent cohomology groups are
\[
\DD^i_{\sF}(\Perf_X,M) \cong \EExt^{i+1}_{O(X)\ten_RA}(\sF,\sF\ten_A M). 
\]
\end{lemma}

\subsection{Goodwillie's theorem}

The following is \cite[Lemma I.3.3]{goodwillieChern}, a reformulation of Goodwillie's theorem on nilpotent ideals (\cite[Theorems II.5.1 and IV.2.6]{goodwillieHCderivations}):
\begin{theorem}\label{etalekey}
 If $S \to T$ is a map of simplicial $\Q$-algebras such that $\pi_0S \to \pi_0T$ is a nilpotent extension, then the map
\[
 \HP^{\Q}(S)\to \HP^{\Q}(T)       
\]
 is a quasi-isomorphism of chain complexes.     
\end{theorem}

\begin{proposition}\label{Jhgslemma}
 The functor $J_X$ from  $s\CAlg_R$ to chain complexes is homotopy-homogeneous.
\end{proposition}
\begin{proof}
The chain complexes $\HC^{R_n}(O(X)_n)$ of flat $R_n$-modules satisfy $\HC^{R_n}(O(X)_n)\ten_{R_n}A_n \cong \HC^{A_n}(O(X_A)_n)$ by construction, and lie in non-negative degrees. These tensor products are thus derived tensor products, and since $\HC$ is concentrated in non-negative chain degrees, 
 the product total complex in our definition of $\HC^{R}(X)$ and $\HC^{A}(X_A) $ is just a total direct sum. 
Thus $\HC^A(X_A)\simeq \HC^R(X)\ten_R^{\oL}A$ for all $A \in s\CAlg_R$, 
which ensures that the functor $A \mapsto  \HC^A(X_A)$ is homotopy-homogeneous by exactness of derived tensor products. 

Take a morphism  $C \to B$ and a    nilpotent extension $A \to B$  in $s\CAlg_R$ with kernel $I$.
Now, since $A\by_BC \to C$ is a nilpotent extension, having kernel $I \by \{0\}$,
Theorem \ref{etalekey} gives  quasi-isomorphisms
\[
 \HP^{\Q}(X_{ A\by_BC})\to \HP^{\Q}(X_C),  \quad \HP^{\Q}(X_A)\to \HP^{\Q}(X_B).   
\]
Thus $\HP^{\Q}(X_{ A\by_BC})$ is trivially quasi-isomorphic to the cocone of 
\[
 \HP^{\Q}(X_C)  \oplus \HP^{\Q}(X_A) \to \HP^{\Q}(X_B),    
\]
i.e.\ to the homotopy fibre product $\HP^{\Q}(X_A)\by^h_{\HP^{\Q}(X_B)} \HP^{\Q}(X_C)$, so the functor $A \mapsto \HP^{\Q}(X_A)$ is also homotopy-homogeneous. The result for $J_X$ now follows by taking homotopy fibres.
\end{proof}

\begin{lemma}\label{Jtgtlemma}
 For all $[f] \in \H_0(J_X(A))$, the tangent space $T_f( \tau_{\ge 0}(J_X,M))$ is canonically quasi-isomorphic to $\tau_{\ge 0}(\HC^R(X)\ten_{R}^{\oL}M[-1])$. 
 
If $O(X)$ is commutative, then  the tangent space $T_f( \tau_{\ge 0}(J_X^p,M)$ is canonically quasi-isomorphic to $\tau_{\ge 0}(\HC^R(X)^{(p-1)}\ten_{R}^{\oL}M[-1])$.
\end{lemma}
\begin{proof}
Since $A \oplus M \to A$ is a nilpotent extension, substituting in Theorem \ref{etalekey} gives
\begin{align*}
J_X(A \oplus M)&=  \cocone\left( \HP^{\Q}(X_{A\oplus M}) \xra{\mathsf{S}}  \HC^{A\oplus M}(X_{A\oplus M})[-2]\right),\\
&\xla{\sim} \cocone\left( \HP^{\Q}(X_A) \xra{\mathsf{S}}  \HC^R(X)\ten_R^{\oL}(A\oplus M) [-2]\right),\\
&\simeq J_X(A) \oplus (\HC^R(X)\ten_R^{\oL}M[-1]),
\end{align*}
and similarly for $J_X^p$.
\end{proof}


\subsection{The semiregularity map}  
\begin{definition}
Given a simplicial $A$-algebra $S$ with ideal $J$, such that $S$ and $S/J$ are both levelwise flat over $A$, write $\HC^A(S\to S/J):= \cocone(\HC^A(S) \to \HC^A(S/J))$, and define $\HH^A(S\to S/J)$ and $\HN^A(S\to S/J)$ similarly.     
\end{definition}

Following the convention that weights are additive with respect to tensor products, a grading $\cW$ on  a ring  compatible with the multiplication  induces gradings $\cW$ on the various cyclic homology complexes as in \cite[\S 9.9]{W}, 
with all natural constructions automatically preserving the gradings. 

Take a  simplicial commutative $\Q$-algebra $A$, a simplicial $A$-module $M$, 
and a (possibly non-commutative) simplicial $A$-algebra $E$, with $M$ and $E$ levelwise flat over $A$. Then we may put a grading on  the trivial square-zero extension $C:= A\oplus M$ of $A$  by setting $A$ to have weight $0$ and $M$ to have weight $1$.  Also setting $E$ to have weight $0$, we then have the following lemma. 

\begin{lemma}\label{kunnethlemmaa}
There is a commutative diagram, for  $\ten = \ten_A$,
\[
 \begin{CD}
   \HC^A(E)\ten  \cW_1C      @<{\mathsf{I}_E \ten \id}<< \HH^A(E)\ten  \cW_1\HH^A(C)^{(0)} @>{\id\ten \mathsf{B}_C}>{\sim}>  \HH^A(E)\ten  \cW_1\HH^A(C)^{(1)}[1]\\
            @|                              @V{\mathsf{I}_{E\ten C}}V{\sim}V       @A{( \id \ten \pr_{\HH^{(1)}}) \circ \pi_{\HH} }A{\sim}A \\
  \cW_1\HC^C(E\ten C)    @<{\phi}<<   \cW_1\HC^A(E\ten C) @>{\mathsf{B}_{E\ten C}}>{\sim}>  \cW_1\HN^A(E\ten C)[1],
   \end{CD} 
\]
where $\sim$ denotes quasi-isomorphism, $\phi$ is the map induced by base change $A \to C$, and the map on the right combines  $\pi_{\HH}\co \HN \to \HH$ with  the K\"unneth isomorphism $\HH^A(E\ten_AC) \cong \HH^A(E)\ten_A\HH^A(C)$ and projection $\pr_{\HH^{(1)}} \co \HH^A(C) \to \HH^A(C)^{(1)}$. 
\end{lemma}
\begin{proof}
The first  square commutes by compatibility of $\mathsf{I}$ with base change, since $\HH^C(C)=C$.
For commutativity of the second square, observe that by definition
\[
\pi_{\HH} \circ \mathsf{B}_{E\ten C} \circ \mathsf{I}_{E\ten C}= \mathsf{B}_{E\ten C} \co \HH^A(E\ten C) \xra{\mathsf{B}_E\ten \id + \id \ten \mathsf{B}_C} \HH^A(E\ten C),
\]
with projection to $\HH^A(C)^{(1)}$ then killing the identity term on $\HH^A(C)^{(0)}$.

By Goodwillie's Theorem over $A$ (in the form of \cite[Corollary 9.9.2]{W}, based on the proof of \cite[Theorem II.5.1]{goodwillieHCderivations}),  $\cW_i\HP^A(E\ten C)$ is acyclic for all $i>0$, 
so  the map $\mathsf{B}_{E\ten C}\co \HC^A(E\ten C) \to \HN^A(E\ten C)[1] $ is a quasi-isomorphism in all non-zero weights by the $\mathsf{SBI}$ sequence. 

Since we discard higher weights, the  quasi-isomorphism  $\mathsf{B}_C \co \cW_1\HH^A(C)^{(0)} \to \cW_1\HH^A(C)^{(1)}[1]$  can be seen from  the HKR equivalence for $\HH^A(\Symm_AM)$, via which it corresponds to the de Rham derivative $d \co \cW_1C \xra{\sim} \cW_1\Omega^1_{C/A}$.

It thus remains only to show that the middle map $\mathsf{I}_{E\ten C}$ is a quasi-isomorphism.
Writing $\HC:=\HC^A$, $\HH:= \HH^A$, we can characterise it 
as the composite  of the quasi-isomorphisms
\begin{align*}
 \HH(E) \ten M &\xra{(\mathsf{I}_E,0)\ten \id} \cocone\left(\HC(E) \xra{\mathsf{S}_E} \HC(E)\right)\ten M \\
 &\to \cocone\left(\HC(E) \ten \cW_1\HC(C) \xra{\mathsf{S}_E\ten \id- \id\ten \mathsf{S}_C} \HC(E) \ten \cW_1\HC(C)[-2]\right)\\
 &\simeq \cW_1\HC(E\ten C).
\end{align*}
Here, the  first map is a quasi-isomorphism by the $\mathsf{SBI}$ sequence.  The second map is well-defined because $\mathsf{S}$ vanishes on  $\HC(C)^{(0)}$, with the natural isomorphism $M \cong \cW_1\HC(C)^{(0)}$  inducing a quasi-isomorphism  $M \to \cW_1\HC(C)$ by \cite[Exercise 9.9.1]{W}.  The final quasi-isomorphism is from the K\"unneth formula for cyclic homology \cite[Corollary 4.3.12]{lodayCyclic}.
\end{proof}

\begin{lemma}\label{kunnethlemma} 
In the setting of Lemma \ref{kunnethlemmaa},
the obvious map
\begin{align*}
\bar{\phi} \co   \HC^A(E\ten_AC \to E)\to &\cocone\left(\HC^C(E\ten_AC)\to \HC^A(E)\right) \\  
  &\cong \cocone\left(\HC^A(E)\ten_AC\to \HC^A(E)\right)
  \simeq \HC^A(E)\ten_A M      
\end{align*}
 is naturally homotopic to the composite 
\begin{eqnarray*}
 \HC^A(E\ten_AC\to E) &\xra[\sim]{\mathsf{B}}& \HN^A(E\ten_AC\to E)[1]\\
&\xra{\pi_{\HH}}&   
 \HH^A(E\ten_AC\to E)[1]
\cong  \HH^A(E)\ten_A \HH^A(C\to A)[1]\\
&\xra{\mathsf{I}\ten\delta}&  \HC^A(E)\ten_AM.
\end{eqnarray*}
where $\delta \co \HH^A(C\to A)[1] \to M$ is the map induced by the canonical derivation $\Omega_{C/A}^1 \cong M$. 
\end{lemma}
\begin{proof} 

Since $\HC(E)\ten M$ has weight $1$, we may restrict to $\cW_1$. In particular, $\bar{\phi}$ is given by the 
natural map $\phi$ from Lemma \ref{kunnethlemmaa} composed with projection to $\cW_1$

The statement now follows by substitution in Lemma \ref{kunnethlemmaa}, once we  note that $\delta$ is left inverse to the quasi-isomorphism $\mathsf{B}_C \co \cW_1\HH^A(C)^{(0)} \to \cW_1\HH^A(C)^{(1)}[1]$.
\end{proof}

\begin{definition}\label{lefschetzdef} 
Following \cite[\S 4.3]{BresslerNestTsygan}, 
given a levelwise flat simplicial $A$-algebra $E$ and a 
 perfect $E$-complex $\sF$, define  the \emph{Lefschetz map} $\cL_{E/A} \co \oR\HHom_E(\sF, \sF) \to \HH^A(E)$ to be given by 
\begin{align*}
 \oR\HHom_E(\sF, \sF) \simeq  \sF\ten^{\oL}_E \oR\HHom_E(\sF,E)
\xla[\beta]{\sim} &(\sF\ten^{\oL}_A \oR\HHom_E(\sF,E))\ten^{\oL}_{E\ten_A^{\oL}E^{\op}} E \\
&\xra{\ev \ten \id} E \ten^{\oL}_{E\ten_A^{\oL}E^{\op}} E,     
\end{align*}         
where $\ev \co \sF\ten^{\oL}_A \oR\HHom_E(\sF,E)) \to E$ is the evaluation map.
\end{definition}

Derived Morita invariance of Hochschild homology  
\cite[Theorem 5.2]{kellerModelDGCat} 
gives a quasi-isomorphism $  \HH^A(E) \to \HH^A( \per_{dg}(E)) $, where $\per_{dg}(E)$ is the dg category of cofibrant perfect $E$-complexes given by applying \cite[\S 4.6]{kellerModelDGCat} to the dg algebra  given by the Dold--Kan normalisation of $E$ equipped with the Eilenberg--Zilber shuffle product. Writing $\cP(X,Y):= \HHom_E(Y,X)$, the bar construction
 $\HH^A( \per_{dg}(E))$ contains, for $\sF$ cofibrant, a subcomplex given by the total complex of
\begin{align*}
& \ldots \to \left(\cP(\sF,E)\ten_A\cP(E,E)^{\ten_A^{n-1}}\ten_A\cP(E,\sF)\right)\oplus \cP(E,E)^{\ten_A^{n+1}}  \to \ldots\\
 &\ldots\to 
 \left(\cP(\sF,E)\ten_A\cP(E,\sF)\right) \oplus \left(\cP(E,E)\ten_A\cP(E,E)\right)  \to
   \cP(\sF,\sF) \oplus\cP(E,E), 
\end{align*}
which is  a quasi-isomorphic subcomplex since it is just a copy of 
\[
 \cone\left(\HH^A(E, \HHom_E(E,\sF)\ten_A\HHom_E(\sF,E)) \xra{(\beta,\HH^A(\ev))} \HHom_E(\sF, \sF) \oplus \HH^A(E)\right)
\]
and $\beta$ is a quasi-isomorphism. It thus follows that the Lefschetz map can alternatively be characterised as the composite $\HHom_E(\sF, \sF) \to \HH^A( \per_{dg}(E)) \xla{\sim} \HH^A(E)$. 

\begin{proposition}\label{AJtgtprop}
Given $\sF \in \Perf(X)$, the tangent map 
\begin{align*}
 T_{\sF}(\Xi, M) \co T_{\sF}(\Perf_X,M) &\to T_{\Xi(\sF)}(\tau_{\ge 0}J_X,M)\\ 
 \tau_{\ge 0}(\oR \HHom_{O(X)}(\sF,\sF\ten_R M)[-1]) &\to \HC^R(X)\ten_{R}^{\oL}M[-1] 
\end{align*}
on simplicial $R$-modules $M$ is given by the composite map 
\[
 \oR \HHom_{O(X)}(\sF,\sF\ten_R M) \xra{\cL_{O(X)/R}} \HH^R(X)\ten^{\oL}_RM \xra{\mathsf{I}} \HC^R(X)\ten_{R}^{\oL}M.       
\]
\end{proposition}
\begin{proof} 

 Writing $C= R\oplus M$ and $E=O(X)$, 
 with  $\lambda$ the equivalence  of Lemma \ref{Jtgtlemma} and  $\theta \co \HN^{\Q}(E\ten_{\Q}-) \to J_X$ the canonical map, we have  a homotopy commutative diagram
 \[
  \xymatrix{
  \HC^{\Q}(E\ten_R^{\oL}C \to E)[-1] \ar[rr]^{\mathsf{B}}_{\sim}  \ar[d] & &\HN^{\Q}(E\ten_R^{\oL}C \to E)\ar[dl] \ar[dd]^{\theta} 
  \\  
  \HC^{R}(E\ten_R^{\oL}C \to E)[-1] \ar[d]_{\bar{\phi}} \ar[r]^{\mathsf{B}}_{\sim} & \HN^{R}(E\ten_R^{\oL}C \to E) \ar[dl]^{(\mathsf{I} \ten \delta) \circ \pi_{\HH}}  \\ 
    \HC^R(E)\ten_R^{\oL}M[-1] \ar[rr]^{\lambda}_{\sim} && J_X(C \to R)  
}
\]
for $\bar{\phi}$ and $(\mathsf{I} \ten \delta) \circ \pi_{\HH}$ as in Lemma \ref{kunnethlemma}, which gives commutativity of that triangle; the top quadrilateral commutes by naturality of $\mathsf{B}$, and the outer rectangle commutes up to canonical homotopy by the $\mathsf{SBI}$ sequence. 
 
Now, $\Xi = \theta \circ \ch^-$, and  the composite map $K \xra{\ch^-} \HN \xra{\pi_{\HH}} \HH$ is just the Dennis trace $\ch$, so commutativity of the lower triangle in the diagram above gives
the tangent map of $\Xi$ as the composite
\[
T_{\sF}(\Perf,M) \xra{\ch-\ch(\sF)} \HH^R(E)\ten_R^{\oL} \HH^R(C\to R)   \xra{\mathsf{I}\ten\delta}  \HC^R(E)\ten_R^{\oL}M[-1].
\]

%

Since $T_{\sF}(\Perf_E,M[-1])$ deloops $ T_{\sF}(\Perf_E,M)$, the map is determined by its behaviour on morphisms 
$(\tau_{\ge 0} \oR \HHom_{E}(\sF,\sF\ten_R M))[-1]$.

Take $\sF$ to be cofibrant and write $E'$ for the simplicial ring corresponding to the dg algebra $\tau_{\ge 0} \HHom_{E}(\sF,\sF)$ (taking levelwise flat replacement if necessary).  
The Dennis trace $\ch_{E'}\co \GL_1(E'\ten_RC) \to \HH^R(E'\ten_RC)$ maps $g$ to $g^{-1}\ten g$, so for $\alpha \in E'\ten_RM$ we have
\[
  \delta (\ch_{E'}(1+\alpha)) =   \delta\left((1-\alpha)\ten (1+\alpha)\right)= (1-\alpha)\alpha= \alpha \in \HH^R(E')\ten_RM,
\]
meaning  the composite $\delta \circ \ch_{E'} \co E'\ten_RM[-1] \to \HH^R(E')\ten_R^{\oL}M[-1]$ is just the natural inclusion map.

Morita functoriality allows us to pass from $\ch_{E'}$ to $\ch_E$ by composing with the maps $ \HH^R(E') \to \HH^R( \per_{dg}(E)) \xla{\sim} \HH^R(E)$, yielding the Lefschetz map $\cL_{E/R}$. 
\end{proof}

Combining Lemma \ref{obs} with Proposition \ref{AJtgtprop} (after base change $R \to B$) gives the following:
\begin{corollary}\label{AJtgtcor} 
Given a square-zero extension $I \to A \xra{e} B $ in   $ s\CAlg_R $ and  $\sF \in \Perf_X(B)$, the space $\Perf_X(A)_{\sF}$ of deformations of $\sF$ over $A$ (i.e.\ the homotopy fibre  of  $\Perf_X(A) \to  \Perf_X(B)$ over $\{\sF\}$) fits into a natural
commutative diagram  
\[
\begin{CD}
 \Perf_X(A)_{\sF} @>>> \{\sF\} @>{o_e}>> \tau_{\ge 0}(\oR\HHom_{X_B}(\sF,\sF\ten_B^{\oL}I)[-2])\\
 @V{\Xi}VV @V{\Xi}VV @VV{\mathsf{I} \circ \cL_{O(X_B)/B}}V \\
 J_X(A) @>>> J_X(B) @>{o_e}>> \tau_{\ge 0}(\HC^R(X)\ten_{R}^{\oL}I[-2])
 \end{CD}
\]
of homotopy fibre sequences, for the Lefschetz map $\cL$ of Definition \ref{lefschetzdef}. 
\end{corollary}

\begin{remark}
At this stage, everything being affine, the corollary implies that deformations of objects are unobstructed, simply because in this affine setting the obstruction space $\HC^R_{-2}(X,I)$ is $0$. Taking derived global sections leads to the   constructions  for schemes and stacks in \S \ref{globalsn}, where the relevant space becomes non-trivial. 
 \end{remark}

 \subsection{dg algebras and dg categories}\label{orlovsn}
 
The definitions of $J_X$ and $\Xi$ generalise in the obvious way to dg algebras, and indeed dg categories, in place of simplicial $R$-algebras $\cO(X)$. Although we will not need these results in \S \ref{globalsn}, we now discuss the extent to which the preceding results extend to those settings.

When chain complexes are not bounded below, there are some subtleties in generalising Definition \ref{HCdef}. Spaltenstein $K$-flatness \cite{spaltenstein} is the relevant flatness condition. As in \cite[\S 3.2]{kellerHCRingedSp},
the Hochschild complex $\HH$  has a natural analogue as a double complex, whose direct sum total complex we take to be the Hochschild complex $\HH$. The complexes  $\HC, \HN$ are then obtained from this in the usual way \cite[\S 3.4]{kellerHCRingedSp}, as $\Tot \CC\CC_{**}$ and  $\Tot^{\Pi} \CC\CC_{**}^N$, while $\HP$ has to be taken as a product-sum total complex  $\Tot^{\Pi,\oplus} \CC\CC_{**}^P$ 
to ensure quasi-isomorphism invariance.

Given a dg category $\C$ for which $\Perf_{\C}$ is homotopy-homogeneous,
the proof of Corollary \ref{AJtgtcor} then automatically applies provided the conclusion of  Goodwillie's Theorem holds for $\C$, i.e.\ provided $\HP^{\Q}(\C\ten^{\oL}_RA) \xra{\sim}\HP^{\Q}(\C\ten^{\oL}_RB)$ for all square-zero extensions $e \co A \to B$ of simplicial $R$-algebras. 

From the long exact sequence of homology, the obstruction $\mathsf{I}\cL_{(\C\ten_RB )/B}(o_e(\sF))$ to lifting $\Xi(\sF)$ from $ \pi_0J_{\C}(B)$ to $\pi_0J_{\C}(A)$ is then guaranteed to vanish whenever the natural map $\HC^R_{-2}(\C,I)\to \H_{-1}J_{\C}(A)$ is injective, where $I= \ker e$. Injectivity of $\HC^R_{-2}(\C,I)\to \HC^R_{-2}(\C,A)$ would imply this; in particular, it holds if $R$ is a field and  $\pi_*A \onto \pi_*B$. Thus for any such  dg category $\C$,  the obstruction  $o_e(\sF) $  
 to lifting a perfect $\C\ten_RB$-module $\sF$ from the derived category $\cD(\C\ten_RB)$ to $\cD(\C\ten_RA)$ would then lie in 
 \[
 \ker\left( \mathsf{I}\cL_{(\C\ten_RB )/B} \co \EExt^2_{\C\ten_RB}(\sF,\sF\ten_BI)\to \HC^R_{-2}(\C,I)\right).
 \]

Goodwillie's Theorem and hence Corollary \ref{AJtgtcor} do not  extend to all dg algebras and dg categories, but they are true for dg algebras concentrated in non-negative chain degrees, via  their Quillen equivalence with simplicial algebras given by the Dold--Kan and Eilenberg--Zilber constructions. All variants of cyclic homology are derived Morita invariant,  preserve exact sequences in the Morita homotopy category, and are additive with respect to semi-orthogonal decompositions \cite[Theorem 5.2a]{kellerModelDGCat}. Thus the class of small $\Q$-linear  dg categories  for which the conclusion of Goodwillie's Theorem 
holds is closed under those operations.
As a special case, we have: 

\begin{corollary}\label{orlovcor} 
If $\cN$ is a  geometric noncommutative scheme in the sense  of \cite[Definition 4.3]{orlovSmoothProperNC} over a characteristic $0$ field $k$, then for any square-zero extension $e\co A \to B$ of commutative $k$-algebras with kernel $I$ and for any perfect $\cN\ten_kB$-complex $\sF$,  the obstruction $o_e(\sF)$ to lifting $\sF$ to a perfect $\cN\ten_kA$-complex  
lies in 
\[
\ker\left( \cL_{(\cN\ten_kB )/B} \co \Ext^2_{\cN\ten_kB}(\sF, \sF\ten^{\oL}_BI) \to \HH_{-2}^k(\cN)\ten_{k}I\right).
\]
\end{corollary}
\begin{proof}
Up to derived Morita equivalence, $\cN$ is defined to be an admissible dg subcategory  of  the dg category $\per_{dg}(\sO_Y)$ of perfect complexes on  a smooth proper $R$-scheme $Y$. By \cite[\S 4.4 following Theorem 3.25]{orlovSmoothProperNC}, $\cN$ is thus a smooth proper dg category, so \cite[Lemma 2.8 and Corollary 3.15]{TVdg} imply that $\Perf_{\cN}$ is  a derived Artin $\infty$-stack, hence homotopy-homogeneous. 

Since $\cN$ is a semi-orthogonal summand of $\per_{dg}(\sO_Y)$, \cite[Theorem 5.2a]{kellerModelDGCat} implies that it gives rise to a direct summand on all variants of cyclic homology. 
Because $Y$ is quasi-compact and semi-separated, the Thomason--Trobaugh excision argument of  \cite[\S 5]{kellerHCRingedSp} as summarised in  \cite[Theorem 5.2b,c]{kellerModelDGCat} means we can calculate cyclic homology of $\per_{dg}(\sO_Y)\ten_kA$ using \v Cech complexes, so the map  $\HP^{\Q}(\cN\ten_kA) \to \HP^{\Q}(\cN\ten_kB) $ is a quasi-isomorphism as a consequence of Goodwillie's theorem applied  Zariski locally on $Y$. 

Corollary \ref{AJtgtcor} thus extends to this setting. Since $k$ is a field, the morphism $A \to B$ admits a $k$-linear splitting, so $\H_0J_{\cN}(A) \to \H_0J_{\cN}(B)$ is surjective and $\mathsf{I}\cL(o_e(\sF))=0$. 

Finally, as  in Remark \ref{horproperrk} the Lefschetz principle  makes $\mathsf{I}$ injective for $\per_{dg}(\sO_Y)$. Thus $\mathsf{I}\co \HH_{*}^k(\cN) \to \HC_{*}^k(\cN)$ is also injective, being a restriction to  direct summands, so $\cL(o_e(\sF))=0$.
\end{proof}

\begin{remark}
Until the final step, the proof of Corollary \ref{orlovcor} only requires that $\cN$ be a smooth proper dg subcategory of  $\per_{dg}(\sO_Y)$, for $Y$ quasi-compact and semi-separated. In such cases, we can still say  that $\mathsf{I}\cL(o_e(\sF))=0$.

We can even replace $\sO_Y$ with  any presheaf $\sA$ of simplicial associative $R$-algebras  on the site of affine open subschemes of  $Y$ for which $\pi_*\sA$ admits  a quasi-coherent $\sO_Y$-algebra structure. In particular,  we can take $\sA$ to be the structure hypersheaf $\sO_{Y'}$ of any derived scheme $Y'$ with underived truncation $Y$, using \cite[Theorem \ref{stacks-lshfthm}]{stacks2}, or to be any  quasi-coherent simplicial associative $\sO_{Y'}$-algebra.

The extension to $\sA$ follows because the  excision argument of  \cite[\S 5]{kellerHCRingedSp}, implying Zariski descent for $\HP$, just relies on  the functors $ j^* \co \per_{dg}(\sA(U)) \to \per_{dg}(\sA(V))$ being dg quotient maps 
whose dg kernels depend only on the complement $U \setminus V$. This follows from
\cite[Proposition 1.4ii]{drinfeldDGQuotient}
and the isomorphism $\H_*(j^*M) \cong \sO(V)\ten_{\sO(U)}\H_*M$ given by the algebraic Eilenberg--Moore spectral sequence.

\end{remark}

\section{The Abel--Jacobi map for (derived) schemes and stacks}\label{globalsn} 
From now on, all rings will be commutative.

\subsection{Derived de Rham cohomology}

\begin{definition}
Given $A \in s\CAlg_{\Q}$  and   $B \in s\CAlg_A $, define the de Rham complex to be the chain complex
\[
\DR(B/A):= \prod_{n\ge 0} \Omega^n(B/A)[n]= \prod_{n\ge 0} (\L^n_B\Omega^1(B/A))[n],
\]
with differential given by combining the differentials on the chain complexes $\Omega^n(B/A)$ with the de Rham differential, i.e.\ the derivation induced by $d \co B \to \Omega^1(B/A)$. This has a Hodge filtration given by
$F^p\DR(B/A):= \prod_{n\ge p} \Omega^n(B/A)[n].$ 
\end{definition}

Beware that the de Rham complex is usually regarded as a cochain complex, so negative homology groups will correspond to positive cohomology groups; we are using chain complexes to facilitate comparison with cyclic homology. 

The functors $\Omega^n(-/A)$ clearly preserve weak equivalences between cofibrant (and indeed levelwise ind-smooth) simplicial $A$-algebras, and the product total complex sends column quasi-isomorphisms of second quadrant double complexes to quasi-isomorphisms (see e.g. \cite[Acyclic Assembly Lemma 2.7.3]{W}), so $\DR(-/A)$ preserves weak equivalences between such objects.

\begin{definition}\label{affLDRdef}
 Given $A \in s\CAlg_{\Q}$  and    $B \in s\CAlg_A $, 
 define the left-derived de Rham complex $\oL\DR(B/A)$ by first taking a cofibrant replacement $\tilde{B} \to B$ over $A$ in the model structure of \cite[\S II.4]{QHA}, 
 then setting
\[
\oL\Omega^p(B/A) := \Omega^p(\tilde{B}/A), \quad \oL\DR(B/A):= \DR(\tilde{B}/A).
\]
Note that this is well-defined up to quasi-isomorphism,  that such replacements can be chosen functorially in both $B$ and $A$, and that  $\oL\Omega^1(B/A)$ is a model for the cotangent complex $\bL(B/A)$.

The complex $ \oL\DR(B/A)$ has a Hodge filtration $F^p\oL\DR(B/A):= F^p\DR(\tilde{B}/A)$, and we  write $\oL\DR(B/A)/F^p:=\oL\DR(B/A)/F^p\oL\DR(B/A)$.
\end{definition}
 
\begin{remark}\label{DixExp}
Following the ideas of \cite{Gr} as developed in \cite{Simfil,GaitsgoryRozenblyumCrystal}, there is a more conceptual interpretation of the derived de Rham complex. 

For any functor $F$ on $s\CAlg_A$, we may define $F_{\inf}(C):= F((\pi_0C)^{\red})$, and then $ F_{\mathrm{strat}}(C) := \im(\pi_0F(C) \to F_{\inf}(C))$ whenever images make sense. 
Note that if  $F$ is  represented by a smooth algebraic space, then $F_{\mathrm{strat}}=F_{\inf}$.

Now for $A \to B$ as above, 
derived $\Hom$ in the model category $s\CAlg_A $ gives a simplicial set-valued functor $\oR \Spec B = \HHom_{s\CAlg_A}(\tilde{B},-)$ on $s\CAlg_A$, with $(\oR \Spec B)_{\inf}(C) \cong \Hom_{\CAlg_{\pi_0A}}(\pi_0B,(\pi_0C)^{\red} )$.

Now, $F_{\mathrm{strat}}(C)$ is equivalent to the \v Cech nerve of $F(C)$ over $F_{\inf}(C)$, which for $F= \oR \Spec B$ is represented  in level $n$ by formal completions of the diagonal map $\tilde{B}^{\ten_A(n+1)} \to \tilde{B}$. Homology of symmetric powers then shows that $\oL\DR(B/A)\simeq \oR \Gamma( (\oR \Spec B)_{\mathrm{strat}}, \O)$, where $\O$ is the hypersheaf given by $\O(C)=C$.
\end{remark}


\begin{proposition}\label{cfDRHC}
For a levelwise flat morphism $A \to B$ in $s\CAlg_{\Q}$ and for all $p \in \Z$, there are canonical quasi-isomorphisms
\begin{align*}
&\HP^A(B)^{(p)} \simeq \oL\DR(B/A)[-2p],\quad  &&\HC^A(B)^{(p)} \simeq (\oL\DR(B/A)/F^{p+1})[-2p],\\  &\HN^A(B)^{(p)} \simeq F^p\oL\DR(B/A)[-2p], \quad  &&\HH^A(B)^{(p)} \simeq \oL\Omega^p(B/A)[-p],  
\end{align*}
with the $\mathsf{SBI}$ sequences corresponding to the short exact sequences $0 \to F^p\oL \DR \to \oL\DR \to \oL\DR/F^p\to 0$ and $0 \to \oL\Omega^p[-p] \to \oL\DR/F^{p+1}\to   \oL\DR/F^p\to 0 $.
       %
       \end{proposition}
\begin{proof}
When $A$ and $B$ are concentrated in degree $0$, with $B$ smooth over $A$, this is a well-known consequence of the Hochschild--Kostant--Rosenberg theorem, as in \cite[Theorem 9.8.13]{W}, and that proof immediately generalises to filtered colimits of smooth morphisms, including infinite polynomial algebras. As observed in \cite[\S 5]{majadas}, the general case for $\HH$ and  $\HC$ then follows by taking a cofibrant replacement for $B$ and passing to the total complex 
of the resulting bisimplicial diagram. The expressions for $\HN$ and $\HP$ follow similarly because product total complexes respect the relevant quasi-isomorphisms.     
\end{proof}

\begin{remark}\label{algDRrmk}
Note that \cite[Theorem 2.2]{emmanouil} (following  \cite[Theorem 5]{FeiginTsygan}) shows that for a  finitely generated algebra $B$ over a characteristic zero field $k$, the complex $\HP^k(B)^{(p)}[2p]$ is quasi-isomorphic to the infinitesimal cohomology complex, or equivalently to Hartshorne's algebraic de Rham cohomology \cite{HartshorneDRcoho} over $k$. For an alternative proof, see \cite{bhattDerivedDR}. 
\end{remark}

\subsection{Global constructions} 

As in \cite{weibelHCschemes}, we now use naturality of the affine constructions for cyclic homology to pass from local to  global. 

Fixing terminology, we will refer to a simplicial set-valued functor on $s\CAlg_R$ as a derived  $\infty$-stack over $R$ if it preserves weak equivalences and  satisfies \'etale hyperdescent, for \'etale morphisms in the sense of \cite[Theorem 2.2.2.6]{hag2}; this is a $D^-$-stack in the terminology of \cite[Definition 2.2.2.14]{hag2}. 
Derived affine schemes of the form $\oR \Spec A$ for $ A \in s\CAlg_R$ (see Remark \ref{DixExp}) are known as derived affine schemes. 
%

\subsubsection{Perfect complexes}

\begin{definition} 
Given a 
derived  $\infty$-stack $X$ over $\Q$,  define the simplicial set
$\Perf(X)$ to be the space $\oR\Gamma(X, \Perf)$ of maps from $X$ to $\Perf$ in the simplicial category of simplicial set-valued  functors on $s\CAlg_R$.
\end{definition}
Note that this is consistent with Definition \ref{perfdef} when $X$ is a derived affine scheme, by the model Yoneda lemma. 

\begin{remarks}\label{hyperdescent}
Using the explicit hyperdescent formulae of  \cite[Examples \ref{stacks-categs}]{stacks2}, when $X$ is a strongly quasi-compact derived Artin $n$-stack the simplicial semiring $\Perf(X)$ can be constructed as follows. First, \cite[Theorem 4.7]{stacks2} provides the existence of a suitable  resolution of the \'etale hypersheaf  $X$ by a derived Artin hypergroupoid $X_{\bt}$, which is a  simplicial derived affine scheme satisfying properties analogous to those of a Kan complex \cite[Examples 3.5]{stacks2}. 
We then define a cosimplicial simplicial semiring given by $\CC^n(X_{\bt},\Perf(\O_{X})):= \Perf(X_n)$, and set
\[
 \Perf(X)=\oR\Tot_{s\Set} \CC^{\bt}(X_{\bt}, \Perf(\O_{X})),
\]
where $\oR\Tot_{s\Set}$ is the derived total functor from cosimplicial simplicial sets to simplicial sets, as in \cite[\S VIII.1]{sht}.

Because $\Perf$ forms an \'etale  hypersheaf, the definition of $\Perf(X)$ above agrees with the standard definition for underived schemes, and indeed for algebraic stacks.
 In the case when $X$ is a quasi-compact semi-separated scheme, $X_{\bt}$ can just be constructed by taking the \v Cech nerve of an affine cover, in which case $\CC^{\bt}$ is just a \v Cech complex.

In our main applications in \S \ref{redobsn}, $X$ will be of the form $Y\ten_RA$, for $A \in s\CAlg_R$ and  $Y$ a smooth quasi-compact semi-separated scheme over a Noetherian ring $R$. For such applications, we can regard $X$ as being the derived scheme associated to a dg scheme (or even a dg manifold) in the sense of \cite{Quot} (for instance by the construction of \cite[\S 6.4]{stacks2}). Then $\Perf(X)$ corresponds to the space of compact objects in the dg derived category of $A$-modules in complexes of quasi-coherent sheaves on $Y$, via the equivalences summarised in \cite[Remark 5.32]{2021lect}.
\end{remarks}

\subsubsection{Derived de Rham cohomology}

We begin with  very general definitions; readers interested in nothing more exotic than a derived Artin $n$-stack can jump straight to Lemma \ref{DRresnlemma} and take it as a definition. 

\begin{definition}
 Given a derived $\infty$-stack $X$ over $R$, we  write $F^p\oL\DR(X/R):=\oR\Gamma(X, F^p\oL\DR(\sO/R))$, where $\sO_X$ is the \'etale  hypersheaf  $\oR\Spec A \mapsto A$ on the site of derived affine $R$-schemes over $X$. Equivalently, $F^p\oL\DR(X/R)$ is the homotopy end $\int^h_{A \in s\CAlg_R} F^p\oL\DR(A/R)^{X(A)}$.
 \end{definition}

  For such $X$ and $A \in s\CAlg_R$, we then write $X_A$ for the derived $\infty$-stack $X$ over $A$ given by base change (i.e. composition with the forgetful functor $s\CAlg_A \to s\CAlg_R$). 
 
 More generally, given a morphism $X \to Y$ of derived $\infty$-stacks, we can  define the \'etale hypersheaf  $X_{\sO_Y}$ (taking values in derived $\infty$-stacks over $\sO_Y$) on the site of derived affine $R$-schemes $V \simeq \oR\Spec A$ over $Y$ by letting $X_{\sO_Y}(V)$ be the derived $\infty$-stack over $A$ sending $B$ to $\oR\Map_Y(\oR \Spec B, X)$
 
 \begin{definition}
 Given a morphism $X \to Y$ of derived $\infty$-stacks over $R$, set $F^p\oL\DR(X/Y):= \oR\Gamma(Y, F^p\oL\DR(X_{\sO_Y}/\sO_Y))$, the $\oR\Gamma(Y, \sO_Y)$-module of derived global sections of the $\sO_Y$-module $F^p\oL\DR(X_{\sO_Y}/\sO_Y)$ in unbounded complexes. 
 
 Equivalently,  for the functors $F^p\oL\DR(-/-)$ and $(X/Y)\co (A \to B) \mapsto X(B)\by^h_{Y(B)}Y(A)$ on  the arrow category $s\CAlg_{R}^{[1]}$, we can interpret $F^p\oL\DR(X/Y)$ as derived global sections of $(X/Y)$ with values in $F^p\oL\DR$, i.e. as the homotopy end
$
 \int^h_{(A \to B)} F^p\oL\DR(B/A)^{(X/Y)(A \to B)} 
$.
 \end{definition}

For morphisms $X \to Y$ of strongly quasi-compact derived Artin $n$-stacks, the following lemma allows us to express $F^p\oL\DR(X/Y)$ simply in terms of the algebraic complexes from Definition \ref{affLDRdef}, since double application of  \cite[Theorem 4.7]{stacks2} gives resolutions by derived affine schemes of the required form. When $Y$ itself is a derived affine scheme, we can take $\tilde{Y}_{\bt}=Y$ and the condition is then just that $\tilde{X}_{\bt}$ be a hypercover of $X$. 

We could alternatively take this formula as a definition, observing that since hypercovers yield quasi-isomorphisms of such complexes, it defines a functor on the relative category of strongly quasi-compact 
derived  Artin $n$-stacks, by \cite[Theorem 6.11]{2021lect}.

\begin{lemma}\label{DRresnlemma}
 Given a morphism $X \to Y$ of derived $\infty$-stacks over $R$, a simplicial  hypercover $\tilde{Y}_{\bt} \to Y$ for the \'etale topology, and a relative simplicial hypercover $\tilde{X}_{\bt} \to X\by^h_Y\tilde{Y}_{\bt}$, the complex $F^p\oL\DR(X/Y)$ is quasi-isomorphic to the product total complex of the double complex $n \mapsto F^p\oL\DR(\tilde{X}_n/\tilde{Y}_n)$.
 \end{lemma}
\begin{proof}
Given a morphism from $(A \to B)$ to $(A' \to B')$ in  $s\CAlg_{R}^{[1]}$ with $A \to A'$ and $B \to B'$ \'etale, we have a natural quasi-isomorphism $\oL\Omega^p(B/A)\ten_B^{\oL}B' \to \oL\Omega^p(B/A)$. Thus the functor  $\oL\Omega^p$ from the arrow category $s\CAlg_{R}^{[1]}$ to the category of $R$-linear unbounded complexes  satisfies  hyperdescent with respect to the topology $\tau$ generated by finite families $\{(A \to B) \to (A'(i) \to B'(i))\}_i$ for which the maps $A \to A'(i)$ and $B \to B'(i)$ are \'etale and $\coprod_i \Spec \pi_0 B'(i) \to \Spec \pi_0B$ is surjective. 
By taking homotopy limits of finite extensions, it follows that the functor  $\oL\DR$ on the arrow category also satisfies $\tau$-hyperdescent. 

The question thus reduces to showing that the simplicial diagram $n \mapsto (\tilde{X}_n/\tilde{Y}_n)$ is a $\tau$-hypercover of $(X/Y)$. This in turn reduces to showing that if $Y' \to Y$ and $X' \to X\by_YY'$ are \'etale local epimorphisms, then $(X'/Y') \to (X/Y)$ is a $\tau$-local epimorphism, but this follows easily from the factorisation
\begin{align*}
  X'(B)\by^h_{Y'(B)}Y'(A) \to &(X\by^h_YY')(B)\by^h_{Y'(B)}Y'(A)\\
  &\simeq X(B)\by^h_{Y(B)}Y'(A) \to X(B)\by^h_{Y(B)}Y(A). \qedhere
\end{align*}
\end{proof}

\begin{definition} 
Given an Artin $n$-stack $X$ locally of finite presentation over $\Cx$, we can form a resolution $\tilde{X}_{\bt}$ of $X$ by l.f.p.  $\Cx$-schemes as in \cite[Theorem 4.7]{stacks2}, giving a simplicial topological space $\tilde{X}_{\bt}(\Cx)_{\an}$ on taking the analytic topology. 
Since $\tilde{X}_{\bt}$ is unique up to smooth hypercovers and since hypercovers induce equivalences on categories of local systems and their cohomology, the homotopy type of the  homotopy colimit is independent of choices made. It can be realised by the
fat geometric realisation $|\tilde{X}_{\bt}(\Cx)_{\an}|$, and we denote it simply by $|X(\Cx)_{\an}|$. When $X$ is  an algebraic space, note that this just recovers $X(\Cx)_{\an}$. 

Given a derived Artin $n$-stack $X$ over $\Cx$ whose underived truncation $\pi^0X$ is locally of finite type, we write $|X(\Cx)_{\an}|:= |(\pi^0X)(\Cx)_{\an}|$; the notation is justified because $X$ and $\pi^0X$ have the same $\Cx$-points.
\end{definition}

When working over Artinian simplicial $\Cx$-algebras (i.e.\ simplicial $\Cx$-algebras which are levelwise Artinian  with bounded Dold--Kan normalisation), we  can  compare derived de Rham cohomology with Betti cohomology of this analytic space of $\Cx$-points: 
\begin{lemma}\label{horblochlemma}
 For $X$ a derived Artin $n$-stack  over a local Artinian simplicial $\Cx$-algebra $A$, 
 with underived truncation locally of finite type, there is a canonical zigzag of quasi-isomorphisms $\oL\DR(X/A) \simeq \oR\Gamma(|X(\Cx)_{\an}|,A)$. 
\end{lemma}
\begin{proof}
Given  a  cofibrant simplicial $A$-algebra  $B$ with $\pi_0B$ finitely generated, we can write $B$ as the filtered colimit of its levelwise finitely generated  cofibrant $A$-subalgebras $B'$ with $\pi_0B'\cong \pi_0B$.
We then have zigzags 
\[
\oR\Gamma((\Spec \pi_0B)_{\an}, A_n) \to \oR\Gamma(\widehat{(\Spec B'_n)}_{\an}, \Omega^{\bt,\an}_{\sO/A_n})  \la \hat{\Omega}^{\bt}_{B_n'/A_n},
\]
where completions are with respect to the map $B_n' \to \pi_0B$, with $\widehat{(\Spec B_n')}_{\an}$ being   the formal Stein space of  \cite[\S I.6]{HartshorneDRcoho} and $\Omega^{\bt,\an}$  the analytic de Rham cohomology complex; note that $\hat{\Omega}^{\bt}_{B_n'/A_n}$ is Hartshorne's algebraic de Rham cohomology of $\pi_0B$ over $A_n$.
Since everything in sight is flat over the  local Artinian ring $A_n$, to see that these maps are quasi-isomorphisms it suffices to know that they are so after  base change along $A_n \to \Cx$, which they are by  \cite[Theorem IV.1.1]{HartshorneDRcoho}. 

Now consider the zigzag
$\DR(B/A) \la \LLim_{B'} \DR(B'/A) \to \LLim_{B'}\widehat{\DR(B'/A)}   $,
where we write  $ \widehat{\DR(B'/A)}:= \Tot^{\Pi}\hat{\Omega}^{\bt}_{B_{\bt}'/A_{\bt}}$.
Again, these maps are all quasi-isomorphisms, because after base change along $A \to \Cx$ they are all quasi-isomorphic to $\HP^{\Cx}(\pi_0B\ten_{\pi_0A}\Cx)$: the first two cases follow from Proposition \ref{cfDRHC} combined with Theorem \ref{etalekey} (replacing $\Q$ with $\Cx$ in Goodwillie's proof), and the third because  \cite[Theorem 2.2]{emmanouil} (following \cite{FeiginTsygan} and Goodwillie's theorem) gives $\hat{\Omega}^{\bt}_{B_n'/A_n}\ten_{A_n}\Cx \simeq \HP^{\Cx}(\pi_0B\ten_{\pi_0A}\Cx)$ for all $n$ and $B'$.

Putting everything together and taking filtered colimits, we have a canonical zigzag $\DR(B/A) \simeq  \oR\Gamma((\Spec \pi_0B)_{\an}, A)$. Applying this locally to a cofibrant simplicial $A$-algebra resolution $\tilde{\sO}_X$ for $\sO_X$ and taking derived global sections then  gives us the required zigzag $ \oL\DR(X/A) \simeq \oR\Gamma(|X(\Cx)_{\an}|,A).$
\end{proof}

\subsection{Generalised Abel--Jacobi maps}

\begin{definition}\label{JXpglobaldef}
Given a morphism $X \xra{f} Y$ of 
derived $\infty$-stacks
over a simplicial $\Q$-algebra $k$, define
\[
 J^p(X/Y,k):= \cocone\left( \oL\DR(X/k) \to \oL\DR(X/Y) /F^p\right)[-2p].
\]
\end{definition}
Note that  when $X$ is affine, Proposition \ref{cfDRHC} gives an equivalence $J^p(X/R,\Q) \simeq J^p(X/R)$ with the construction of  Definition \ref{JXdef}.



\begin{example}\label{deligneex} 
 If $C \to A$ is a morphism of Artinian local $\Cx$-algebras and $X$ a smooth proper scheme over $A$, then the complexes $J^p(X/A,C)$ admit an underived  analytic description as follows. 

Applying GAGA to graded pieces as in \cite[Proposition 3.8]{blochSemiregularity} gives a quasi-isomorphism $\oR\Gamma(X, \Omega^{\bt}_{X/A}/F^p) \to \oR\Gamma(X(\Cx)_{\an}, \Omega^{\bt}_{X_{\an}/A}/F^p)$, and Lemma \ref{horblochlemma} gives a compatible quasi-isomorphism $\oR\Gamma(X, \oL\DR(\sO_X/C)) \simeq \oR\Gamma(X(\Cx)_{\an},C)$. Combining these, we have
\[
 J^p(X/A,C)  \simeq \oR\Gamma(X(\Cx)_{\an}, C \to \sO_{X/A}^{\an} \xra{d} \Omega^{1,\an}_{X/A}  \xra{d} \ldots \xra{d} \Omega^{p-1,\an}_{X/A})[-2p], 
\]
giving our space an interpretation as a form of Deligne cohomology.
 \end{example}


\begin{definition} \label{JXdef2}
 Given 
a simplicial commutative $\Q$-algebra $R$ and a 
derived $\infty$-stack $X$
over $R$, define $J^p_X$ to be the functor on $s\CAlg_R$ given by $J^p_X(A):= J^p(X_A/A,\Q)$.

Note that Proposition \ref{cfDRHC} ensures this is equivalent to Definition \ref{JXdef} when $X$ is a derived affine scheme represented by a simplicial $R$-algebra $O(X)$ chosen to be levelwise flat. 
\end{definition}

 Since Lemma \ref{Jtgtlemma} and Proposition \ref{cfDRHC} give a canonical zigzag of quasi-isomorphisms between the corresponding affine constructions, passing to homotopy limits gives:
\begin{lemma}\label{Jtgtlemma2}
 For $f \in \H_0(J_X^p(A))$, the tangent space $T_f( \tau_{\ge 0}(J_X^p,M))$ is canonically quasi-isomorphic to $\tau_{\ge 0}(((\oL\DR(X/R)/F^p)\ten^{\oL}_RM)[2p-1] )$. 
\end{lemma}

\begin{remark}
 Although not needed for our applications, if we replace $\Q$ with any simplicial $\Q$-algebra $k$ throughout Definition \ref{JXdef2}, the description of the tangent space in Lemma \ref{Jtgtlemma2} remains valid.
\end{remark}

\begin{definition}\label{generalAJdef} 
Globalising definition \ref{AJdef}, define the Abel--Jacobi map 
\[
 \Xi_{k} \co \Perf(X) \to \tau_{\ge 0}\prod_{p \ge 0} J^p(X/Y,k)
\]
(which we simply denote as $\Xi$ when $k=\Q$) as follows. Via the equivalences of Proposition \ref{cfDRHC}, the Goodwillie--Jones Chern character gives us maps 
\[
\ch_p^- \co \Perf(\sO_{X}) \to \tau_{\ge 0} (F^p\oL\DR(X/k)[-2p]).
\]
We then take derived  global sections $\oR\Gamma(X,-)$ and   compose  with the natural maps
\[
 \oR\Gamma(X,\tau_{\ge 0} (F^p\oL\DR(X/k)[-2p])  ) \to \tau_{\ge 0}(\oR\Gamma(X, F^p\oL\DR(X/k) )[-2p]) \to \tau_{\ge 0}J^p(X/Y,k). 
\]
 \end{definition}

\begin{definition}\label{tgtAJdef}
For $\sF \in \Perf_{X}(A)$ and a simplicial $A$-module $M$,  write
\[
\xi^i\co \EExt^{i+1}_{\sO_{X_A}}(\sF,\sF\ten^{\oL}_A M) \to \prod_{p\ge 0} \H^{2p+i-1}((\oL\DR(X/R)/F^p)\ten^{\oL}_RM) 
\]
 for the tangent map
\[
 \DD^i_{\sF}(\Xi, M)\co \DD^i_{\sF}(\Perf_{X},M) \to \prod_{p\ge 0} \DD^i_{\Xi_p(\sF)}(\tau_{\ge 0}J_{X}^p,M).
\]
\end{definition}

Substituting our hypersheaves in Proposition \ref{AJtgtprop} and taking derived global sections yields: 

\begin{proposition}\label{AJtgtprop2}
The tangent map $\xi_p^i$ is given by composing the Lefschetz map $\cL_{X/R}\co \EExt^{i+1}_{\O_{X_A}
}(\sF,\sF\ten^{\oL}_A M) \to \bH^{p+i}(X, \oL\Omega^{p-1}_{X/R}\ten^{\oL}_RM)$ with the canonical map
$\mathsf{I} \co \bH^{p+i}(X, \oL\Omega^{p-1}_{X/R}\ten^{\oL}_RM)\to \bH^{2p+i-1}(X,(\oL\DR(\O_{X}/R)/F^p)\ten^{\oL}_RM) $.
\end{proposition}

\begin{remark}\label{cfBF}
The construction of the Atiyah--Hochschild character $\AH(\sF)$ of \cite[\S 5]{BuchweitzFlennerGlobal}  just makes it the dual of the Lefschetz map, in the sense that $\cL(\alpha) = \tr( \AH(\sF)\circ \alpha)$. Thus \cite[Theorem 5.1.3 and Proposition 6.2.1]{BuchweitzFlennerGlobal} ensure that $\cL$ is the same as the semiregularity map $\sigma$ of \cite{BuchweitzFlenner}, given by applying the exponential of the Atiyah class then taking the trace.

Thus the semiregularity map $\sigma$ is induced by our Abel--Jacobi map $\Xi_{p,\Cx}$, which Example \ref{deligneex} interprets
as a  Chern character taking values in a form of Deligne cohomology, exactly as anticipated in \cite[\S 1]{BuchweitzFlenner}.
\end{remark}

Combining Lemma \ref{obs} with Proposition \ref{AJtgtprop2} and Remark \ref{cfBF}, or just substituting in Corollary \ref{AJtgtcor} and taking derived global sections, gives the following:
\begin{corollary}\label{AJtgtcor2}
Given a square-zero extension $I \to A \xra{e} B $ in   $ s\CAlg_R $ and an object  $\sF \in \Perf_X(B)$, the space $\Perf_X(A)_{\sF}$ of deformations of $\sF$ over $A$ fits into a natural
commutative diagram  
\[
\begin{CD}
 \Perf_X(A)_{\sF} @>>> \{\sF\} @>{o_e}>> \tau_{\ge 0}(\oR\HHom_{\sO_{X_B}}(\sF,\sF\ten_B^{\oL}I)[-2])\\
 @V{\Xi_p}VV @V{\Xi_p}VV @V{\cL_{p-1}}V{=\sigma_{p-1}}V \\
 J_X^p(A) @>>> J_X^p(B) @>{o_e}>> \tau_{\ge 0}\oR\Gamma(X,(\oL\DR(\O_{X}/R)/F^p)\ten^{\oL}_RI[-2p])
 \end{CD}
\]
of homotopy fibre sequences, for the Lefschetz map $\cL$ of \cite{BresslerNestTsygan}
and semiregularity map  $\sigma$ of \cite{BuchweitzFlenner}. 

Thus the image of $\cL_{p-1}(o_e(\sF))$ in $\bH^{2p}( X,(\oL\DR(\O_{X}/R)/F^p)\ten^{\oL}_RI))$ is the obstruction to lifting $\Xi(\sF)$ from $\H_0J_X(B)$ to $\H_0J_X(A)$.  
\end{corollary}

\begin{remark}\label{splitrmk}
In particular,   if $A \to B$ admits a section in the derived category of $R$-modules (automatic if $R$ is a field and $\pi_*A\onto \pi_*B$), 
then $J_{X}^p(A)\simeq J_{X}^p(B) \oplus (\oL\DR({X}/R)/F^p)\ten_RI$, so 
 $\H_0J_X(A)\to \H_0J_X(B)$ is surjective and  Corollary \ref{AJtgtcor2}
implies that 
\[
o_e(\sF) \in \ker(\cL_{p-1} \co \Ext^2_{\sO_{X_A}}(\sF,\sF\ten^{\oL}_AI) \to  \bH^{2p}( X,(\oL\DR(\O_{X}/R)/F^p)\ten^{\oL}_RI)).
\]
\end{remark}

%
%

\subsection{Horizontal sections}\label{horsn}  


We now introduce variants of the obstruction $\cL_{p-1}(o(\sF))$ which are potentially weaker but tend to be more tractable. 
Our most generally applicable result is the following,
writing $\oL\Omega^j(X/S):=\oR\Gamma(S, \oR\Gamma(X_{\sO_S},\oL\Omega^j_{X_{\sO_S}/\sO_S}))$ and $\oL\sI\Omega^j(X/S):=\oR\Gamma(S, \oR\Gamma(X_{\sO_S},\sI\ten^{\oL}_{\sO_S}\oL\Omega^j_{X_{\sO_S}/\sO_S}))$: 
\begin{corollary}\label{gensemiregcor} 
Take a  morphism $f \co X \to S$ 
of derived $\infty$-stacks over $\Q$ and  a closed immersion $e \co S' \into S$ \cite[Definition 2.2.3.5]{hag2}
defined by an ideal $\sI$ with $\pi_0\sI$  nilpotent in $\sO_{\pi^0S}=\pi_0\sO_S$. 

Then for $X':=X\by_{S}^hS'$, the Goodwillie--Jones Chern character $\ch_p^-$ for $\sO_{X'}$ over $\sO_S$ 
induces a  form of Abel--Jacobi map $\Xi_{p,S}$ from  $K_0(X')$ to 
the cohomology group
\[
 \H^{2p}\Tot^{\Pi}(\oL \sI\Omega^0 \xra{d} \oL \sI\Omega^1 \xra{d} \ldots \xra{d} \oL \sI\Omega^{p-1} \xra{d} \oL\Omega^p     \xra{d} \oL\Omega^{p+1}  \xra{d}\ldots )(X/S).
\]

If $\sI$ is square-zero, then
for  any perfect complex $\sF$ over  $X'$
the image of $\Xi_{p,S}(\sF)$ in $\bH^{2p}\oL \sI\Omega^{< p}(X/S)$ 
 is  given by applying the composite  map 
 \[
  \EExt^{2}_{\sO_{X'}}(\sF,\sF\ten_{\sO_{S'}}^{\oL} \sI) \xra[=\sigma_{p-1}]{\cL_{p-1} } \H^{p+1}\oL\sI\Omega^{p-1}(X/S) \to \bH^{2p}\oL\sI\Omega^{<p}(X/S)
 \]
 to  the obstruction 
 $o_e(\sF)$
 to deforming $\sF$ to an $\sO_X$-module in complexes. 
\end{corollary}
\begin{proof}
The map $\Xi_{p,S}$ is defined as the composite of $\Xi_p$ with the natural map
\begin{align*}
 J^p(X'/S',\Q) &=\cocone( \oL\DR(X'/\Q) \to \oL\DR(X'/S') /F^p)[-2p]\\ 
 &\xla{\sim} \cocone( \oL\DR(X/\Q) \to \oL\DR(X'/S') /F^p)[-2p]\\
 &\to \cocone( \oL\DR(X/S) \to \oL\DR(X'/S') /F^p)[-2p] =: J^p(X'/S',S),
\end{align*}
where the equivalence of the second line follows from Theorem \ref{etalekey} because $\pi^0X' \to \pi^0X$ is a nilpotent thickening.
Since $ \oL\sI\Omega^j(X/S) \simeq \cocone(\oL\Omega^j(X/S) \to \oL\Omega^j(X'/S'))$, this gives the target in the required form.


When $\sI$ is square-zero, substituting in Corollary \ref{AJtgtcor2} and taking derived global sections over $S$ then gives us the first two rows of the  
commutative diagram  
\[
\begin{CD}
 \Perf(X)_{\sF} @>>> \{\sF\} @>{o_e}>> \oR\HHom_{\sO_{X}}(\sF,\sF\ten_{f^{-1}\sO_{S}}^{\oL}f^{-1}\sI)[-2]\\
 @V{\Xi_p}VV @V{\Xi_p}VV @V{\cL_{p-1}}V{=\sigma_{p-1}}V \\
 J^p(X/S,\Q) @>>> J^p(X'/S',\Q) @>{o_e}>> 
 \oL \sI\Omega^{< p}(X/S)[-2p]\\
 @VVV @VVV @| \\
 F^p\oL\DR(X/S)[-2p] @>>> J^p(X'/S',S) @>{o_e}>> \oL \sI\Omega^{< p}(X/S)[-2p]
\end{CD}
\]
of homotopy fibre sequences, where we have omitted $\tau_{\ge 0}$ from all complexes to lighten the notation; the third sequence follows  from the  simple calculation that $\cone( F^p\oL\DR(X/S) \to J^p(X'/S',S)[2p]) \simeq \oL \sI\Omega^{< p}(X/S)$. This diagram yields the required description since the first two composite vertical maps are $\ch_p^-$ and $\Xi_{p,S}$.
\end{proof}

Corollary \ref{gensemiregcor} relates the semiregularity map to the Goodwillie--Jones Chern character. Over $\Cx$, we now give a slightly weaker, but more accessible, statement in terms of the topological Chern character.  

When applied to the case where $X$ is a smooth proper scheme over a local Artinian $\Cx$-algebra $A$,  the following  corollary  establishes the conjectures of \cite{BuchweitzFlenner} and hence those of \cite{blochSemiregularity}.  By analogy with   \cite[3.9]{blochSemiregularity}, we can regard it as saying that the semiregularity map constrains deformations of the Chern character $\ch_p(\sF)$  as a horizontal section in $F^p$.  

Writing $F^p\H^{*}(\oL\DR(X/A)) :=\im\left(\H^{*}(F^p\oL\DR(X/A))\to \H^{*}(\oL\DR(X/A)) \right)$, we have:
\begin{corollary}\label{horizobscor2} 
 Take a  local Artinian simplicial $\Cx$-algebra $A$, a derived Artin $n$-stack $X$  over $A$ whose underived truncation is locally  of finite type, 
 and a square-zero simplicial ideal  $I \subset  A$ with quotient $e \co A \to B$. 
 
 Then for  any perfect complex $\sF$ over  $X':=X\ten_A^{\oL}B$,
 with $o_e(\sF) \in \EExt^{2}_{\sO_{X'}}(\sF,\sF\ten_B I)$ the obstruction to deforming $\sF$ to an $\sO_X$-module in complexes,
 the 
 image of the 
 Chern character $\ch_p(\sF)$ under the  map 
\begin{align*}
 \H^{2p}(|X'(\Cx)_{\an}|,\Q) \cong \H^{2p}(|X(\Cx)_{\an}|,\Q) 
 \to 
 \H^{2p}(|X(\Cx)_{\an}|,A) 
 \cong \H^{2p}(\oL\DR(X/A))
\end{align*}
 lies in $F^p\H^{2p}(\oL\DR(X/A))$ if and only if the image of $\cL_{p-1}(o_e(\sF))$ under the map  $\bH^{p+1}(X,\oL\Omega^{p-1}_{X/A}\ten_A^{\oL}I) \to \bH^{2p}(X,\Tot^{\Pi}(\oL\Omega^{<p}_{X/A},d))$, coming from the inclusion $I \into A$, is zero.
 \end{corollary}
 \begin{proof}
The obstruction $\kappa$ is given by the image of 
$\ch_p(\sF)$ 
in $ \H^{2p}(\oL\DR(X/A)/F^p)$, which 
 Lemma \ref{horblochlemma} tells us is  the image of the  Chern character $\ch_p^{\dR}(\sF)
 \in  \H^{2p}(\oL\DR(X'/\Cx)) \cong \H^{2p}(\oL\DR(X/\Cx))$ under the morphism  
 \[
  \H^{2p}(\oL\DR(X/\Cx)) \to \H^{2p}(\oL\DR(X/A)/F^p).
  \]

  By compatibility of the Goodwillie--Jones Chern character with more classical Chern characters 
  as in \cite{HoodJones} and references therein,  
$\ch_p^{\dR}(\sF)$ is  the image of $\Xi_{p}(\sF)$ under the natural map $J^p(X'/B) \to \oL\DR(X'/\Cx)[-2p]$. It thus follows that $\kappa$ is the image of $\Xi_{p,A}(\sF)$ 
  under the composite map
  \begin{align*}
   \H_0J^p(X'/B,A)= ~&\H^{2p}\cone\left(\oL\DR(X/A) \to \oL\DR(X'/B)/F^p\right) \\
   &\to \H^{2p}\oL\DR(X/A) \to \H^{2p}(\oL\DR(X/A)/F^p), 
  \end{align*}
which is the same as the composite map
 \begin{align*}
  \H_0J^p(X'/B,A) \xra{o_e} &~\H^{2p}\cone\left(\oL\DR(X/A)/F^p \to \oL\DR(X'/B)/F^p\right)\\
 &\cong \bH^{2p}(X,\oL\Omega^{<p}_{X/A}\ten_A^{\oL}I) \to \bH^{2p}(X,\oL\Omega^{<p}_{X/A} ),
 \end{align*}
so the description of $o_e(\Xi_{p,A}(\sF))$ from Corollary \ref{gensemiregcor} completes the proof.
 \end{proof}
 
 \begin{remark}[$\mu$-twisted sheaves]\label{mutwistrmk}
 Because the group scheme $\mu_r$ of $r$th roots of unity is \'etale over $\Q$ and $\mu_r$ is reductive,  $\H^*(B\mu_r, \oL\Omega^p_{B\mu_r/\Q})$ vanishes for $p>0$ and is $\Q$ for $p=0$, so the K\"unneth isomorphism applied to the \v Cech nerve of $\tilde{X}_{\alpha} \to X$ gives $\H^*(\tilde{X}_{\alpha}, \oL\Omega^p
 )  \cong  \H^*(X, \oL\Omega^p
 )$ for all $\mu_r$-gerbes $\tilde{X}_{\alpha}$ over $X$ (i.e. homotopy fibres of maps $\alpha \co X \to B^2\mu_r$). Since $\mu_r$-twisted perfect complexes on $X$ (i.e. maps $X \to [\Perf/B\mu_r]$) give rise to perfect complexes on such gerbes, replacing $X$ with $\tilde{X}_{\alpha}$ in  Corollaries \ref{gensemiregcor} and \ref{horizobscor2}   immediately extends their conclusions to $\mu_r$-twisted perfect complexes $\sF$ on  $X'$ with twist $[\alpha] \in \H^2_{\et}(X', \mu_r) \cong \H^2_{\et}(X, \mu_r)$.
 \end{remark}

\subsection{Reduced obstructions}\label{redobsn}
\begin{remark}\label{horproperrk} 
 Assume that $R$ is a  Noetherian $\Q$-algebra, with $X$ a smooth proper scheme over $\Spec R$. Then the Lefschetz principle and degeneration of the Hodge--de Rham spectral sequence \cite[\S 5--6]{deligneLefschetz} imply there exists an $R$-linear quasi-isomorphism
 $
 \DR(X/R)/F^{p}\simeq  \bigoplus_{i=0}^{p-1} \oR\Gamma(X, \Omega^i_{X/R})[i],
 $
and in particular that the morphism $ \oR\Gamma(X, \Omega^{p-1}_{X/R})[p-1] \to \DR(X/R)/F^p$ admits an $R$-linear retraction.

For any simplicial ideal $I$ in a simplicial $R$-algebra $A$ for which   $\pi_*I \to \pi_*A$ is injective, the map 
\[
 \bH^{p+1}(X,\Omega^{p-1}_{X/R}\ten_R I)\to \bH^{2p}(X, (\DR(\O_X/R)/F^p)\ten_RA)
\]
 is therefore injective. In particular, this means that the element $\cL_{p-1}(o_e(\sF)) $ in $\bH^{p+1}(X,\Omega^{p-1}_{X/R}\ten_R I)$ will vanish provided 
its image in $\bH^{2p}(X, \DR(\O_{X_A}/A)/F^p)$ does so.
 
 In the case where $X$ is a smooth proper scheme over an Artinian $\Cx$-algebra, vanishing of $\cL_{p-1}(o_e(\sF))$ itself is thus equivalent to the conditions of  Corollary \ref{horizobscor2} (taking $R=A$), which we can paraphrase as saying that  $\cL_{p-1}(o_e(\sF))$ is the obstruction to 
the unique horizontal lift of $\ch_p^{\dR}(\sF)$ still lying in $F^p\H^{2p}\DR(X/A)$, or equivalently remaining of pure Hodge type $(p,p)$. 

Taking an open substack $\fM \subset \Perf_X$ for which 
that obstruction
vanishes at all points $[\sF]$
 (for instance by restricting to the Hodge locus of \cite{voisinHodgeLociAbs}), 
 we then have a
 functorial obstruction theory 
 \[
\left([\sF] \in \fM(B)\right)
\mapsto 
\ker\left(\cL_{p-1}\co \EExt^{2}_{\O_{X_B}}(\sF,\sF\ten_B^{\oL}-) \to  \H^{p+1}(X,\Omega^{p-1}_{X/R})\ten_R^{\oL}-\right)
\]
 for $\fM$
 as  a subspace of the standard obstruction theory.

 Such reduced obstruction theories for pairs $(X,\sF)$ are   required in the study of Pandharipande--Thomas invariants.
\end{remark}

\begin{remark}\label{GWrk}
 There is a morphism from the derived moduli stack 
 of proper schemes over $X$ (\cite[Theorem \ref{dmsch-representdaffine}]{dmsch}) to the derived stack $\Perf_X$, given by sending $(f\co Z\to X_B)$ to $\oR f_*\O_Z$.  Since the obstruction maps of Lemma \ref{obs} are functorial, this gives rise to a morphism
\[
\psi\co \EExt^2_{\sO_Z}(\bL^{Z/X_B}, \O_Z\ten_B^{\oL} -) \to \Ext^2_{\sO_{X_B}}(\oR f_*\O_Z,\oR f_*\O_Z\ten_B^{\oL} -) 
\]
of obstruction theories. 

For $X$ smooth and proper over $R$, Remark \ref{horproperrk} then implies that $\cL_{p-1} \circ \psi$ annihilates  obstructions to deforming $Z$ over $X$ provided the unique horizontal lift of  $\ch_p(\oR f_*\sO_Z)$ remains of Hodge type. For any open substack $\fN$ of the moduli stack 
 of proper schemes over $X$  which parametrises schemes  satisfying that condition,
%
%
this gives rise to a reduced global obstruction theory, sending $[Z \to X_B] \in \fN(B)$ to 
\[
\ker\left(\cL_{p-1}\circ \psi \co \Ext^2_{\sO_Z}(\bL^{Z/X_B}, \O_Z\ten_B^{\oL} - ) \to  \H^{p+1}(X,\Omega^{p-1}_{X/R})\ten_R^{\oL} -\right);
\]
in particular this applies to stable curves $Z$ over $X$, as  required in the study of Gromov--Witten invariants (see for instance \cite[\S 2.2]{KoolThomas1}). 
\end{remark}

\begin{remark}\label{Blochrk}
 The proof of Proposition \ref{AJtgtprop} characterises the Lefschetz map $\cL$ as a deformation of the Dennis trace.
This means that for any proper  LCI morphism $f\co Z \to X_B$, the Grothendieck--Riemann--Roch theorem 
allows us to interpret the semiregularity map $\cL \circ \psi$ of Remark \ref{GWrk}  as the deformation of the Todd class $f_*(\Td(T_f))$ as $f$ varies.

 When $Z \subset X$ is a codimension $p$ LCI subscheme of a smooth proper scheme, 
\cite[Proposition 8.2]{BuchweitzFlenner} combines with Remark \ref{cfBF} to show that the map
\[
 \cL_{p-1} \circ \psi\co \H^1(Z,\sN_{Z/X}) \to  \H^{p+1}(X,\Omega^{p-1}_{X})
\]
is just Bloch's semiregularity map 
from \cite{blochSemiregularity}, which  is defined in a relatively elementary way in terms of Grothendieck--Verdier duality.  
\end{remark}


\bibliographystyle{alphanum}
\bibliography{references.bib}

\end{document}